\documentclass[11pt]{amsart}

\usepackage{mathrsfs}
\usepackage{amsmath, amscd, amsthm,amssymb, amsfonts, verbatim,subfigure, enumerate, mathtools}
\usepackage[mathcal]{eucal}

\usepackage[all]{xy}
\usepackage{graphicx}

\date{}

\newcommand{\fbracket}[1]{\left\{ #1 \right\}}
\newcommand{\bracket}[1]{\left( #1 \right)}
\newcommand{\bbracket}[1]{\left[ #1 \right]}
\newcommand{\pa}{\tr}
\newcommand{\Lag}{\mscr{L}}
\newcommand{\dpa}[1]{\frac{\partial}{\partial #1}}
\newcommand{\vac}{\lvert \emptyset \rangle}
\newcommand{\zbar}{\br{z}}
\newcommand{\wbar}{\br{w}}
\newcommand{\Res}{\op{Res}}

\newcommand{\Obs}{\op{Obs}}

\newcommand{\eps}{\epsilon}
\newcommand{\g}{\mathfrak{g}}

\newcommand{\Ool}{\Oo_{loc}}

\newcommand{\E}{\mscr{E}}

\newcommand{\A}{\mscr A}
\newcommand{\what}{\widehat}
\newcommand{\tr}{\triangle}

\newcommand{\til}{\widetilde}
\newcommand{\mscr}{\mathscr}

\newcommand{\br}{\overline}

\newcommand{\iso}{\cong}
\newcommand{\C}{\mathbb C}

\newcommand{\Q}{\mbb Q}

\newcommand{\Oo}{\mscr O}
\newcommand{\Z}{\mathbb Z}

\newcommand{\into}{\hookrightarrow}

\newcommand{\Gr}{\operatorname {Gr}}
\newcommand{\op}{\operatorname}
\newcommand{\mbf}{\mathbf}
\newcommand{\mbb}{\mathbb}
\newcommand{\mf}{\mathfrak}
\newcommand{\mc}{\mathcal}

\newcommand{\ip}[1]{\left\langle #1 \right\rangle}
\newcommand{\abs}[1]{\left| #1 \right|}

\newcommand{\R}{\mbb R}
\renewcommand{\d}{\mathrm{d}}
\renewcommand{\epsilon}{\varepsilon}

\newcommand{\mbar}{\br {\mc M}}

\newcommand{\dbar}{\br{\partial}}
\renewcommand{\Im}{\op{Im}}

\newcommand{\cinfty}{C^{\infty}}
\newcommand{\EA}[1]{ W\left( #1  \right) }
 
\newcommand{\PV}{\mathrm{PV}}
\newcommand{\Fields}{\E}
\newcommand{\F}{\mbf{F}}

 \DeclareMathOperator{\End}{End}
 
\DeclareMathOperator{\Sym}{Sym} \DeclareMathOperator{\Hom}{Hom}

\DeclareMathOperator{\Tr}{Tr} 
\DeclareMathOperator{\Ker}{Ker}

\newtheoremstyle{thm}
  {7pt}
  {7pt}
  {\itshape}
  {}
  {\bf}
  {.}
  {5pt}
  {\thmnumber{#2 }\thmname{#1}\thmnote{ (#3)}}

\newtheoremstyle{def}
  {7pt}
  {10pt}
  {\itshape}
  {}
  {\bf}
  {.}
  {5pt}
  {\thmnumber{#2} \thmname{#1}\thmnote{ (#3)}}

\newtheoremstyle{rem}
  {4pt}
  {7pt}
  {}
  {}
  {\itshape}
  {:}
  {3pt}
  {}

\newtheoremstyle{texttheorem}
  {8pt}
  {8pt}
  {\itshape}
  {}
  {\bf}
  {. \hspace{5pt}}
  {3pt}
  {}

\theoremstyle{thm}

\newtheorem*{theorem*}{Theorem}

\newtheorem*{ucorollary}{Corollary}

\newtheorem{theorem}{Theorem}[subsection]
\newtheorem{thm-def}{Theorem/Definition}[theorem]

\newtheorem{proposition}[theorem]{Proposition}
\newtheorem{lemma}[theorem]{Lemma}

\newtheorem{corollary}[theorem]{Corollary}

\numberwithin{equation}{subsection}

\theoremstyle{def}
\newtheorem{definition}[theorem]{Definition}

\theoremstyle{rem}

\newtheorem*{remark}{Remark}

\theoremstyle{texttheorem}

\title{Quantum BCOV theory on Calabi-Yau manifolds and the higher genus $B$-model}
\author{Kevin Costello and Si Li}
\begin{document}
\maketitle
\section{Introduction}
Since it was first formulated more than twenty years ago \cite{CanOssGre91,GrePle90}, the mirror symmetry conjecture has had a huge influence on mathematics.   The original form of the conjecture states that one can count numbers of rational curves on a Calabi-Yau $3$-fold $X$ by analyzing the variations of Hodge structure of a mirror Calabi-Yau $3$-fold $X^\vee$.   Soon after this form of the conjecture was first formulated, it was rigorously proved by Givental \cite{Giv98} and Lian-Liu-Yau \cite{LiaLiuYau97}.  

This paper is concerned with mirror symmetry at higher genus.   The count of rational curves on a Calabi-Yau variety $X$ is referred to as the genus $0$ $A$-model on $X$.  The theory of Gromov-Witten invariants allows one to define the $A$-model at all genera; the genus $g$ $A$-model is concerned with counts of genus $g$ curves mapping to $X$.  

The $B$-model at genus $0$ on the mirror Calabi-Yau $X^\vee$ is described by the variations of Hodge structure of $X^\vee$ (as $X^\vee$ moves in the moduli space of Calabi-Yau varieties).   The higher genus $B$-model is, however, much more mysterious.

One mathematical approach to the higher genus $B$-model is categorical. Kontsevich's homological mirror symmetry conjecture states that the mirror symmetry phenomenon is explained by an equivalence of Calabi-Yau categories, between the Fukaya category of $X$ and the derived category of coherent sheaves on the mirror $X^\vee$.  From this point of view, it is natural to expect to be able to construct the all-genus $B$-model directly from the category of coherent sheaves on $X^\vee$, or, indeed, from any Calabi-Yau category.  A candidate for the $B$-model partition function associated to a Calabi-Yau category was proposed on \cite{Cos07a, Cos05, KonSoi06} based on a classification of a class of $2$-dimensional topological field theories.   Unfortunately, it is extremely difficult to compute this $B$-model partition function.

In the physics literature, the Bershadsk-Cecotti-Ooguri-Vafa \cite{BerCecOog94} initiated a differrent approach to the $B$-model.    BCOV suggested that we should think of the $B$-model as a \emph{quantum field theory on $X$}.  This is in contrast to most other formulations of mirror symmetry, where both the $A$- and $B$- model are viewed as as two-dimensional field theories.   The genus expansion of the $B$-model should correspond to the expansion in powers of the Planck constant $\hbar$ of this quantum field theory on $X$. The genus $0$ part of the $B$-model on a Calabi-Yau $X$ is concerned with deformations of the complex structure of $X$.  Thus, BCOV proposed a quantum field theory whose classical equations of motion describe the moduli space of deformations of complex structure of $X$.  They called this theory the Kodaira-Spencer theory of gravity.  We will call it the BCOV theory. 

Because the BCOV theory is a six-dimensional quantum field theory, it is very difficult to construct rigorously.  In this paper we initiate a mathematical analysis of the BCOV theory. 

\subsection{}
The original BCOV theory was defined only for Calabi-Yau $3$-folds. The first step in our work is to define a variant of the classical BCOV theory which works for Calabi-Yau varieties of any dimension.  The fields of our generalized BCOV theory on $X$ are
$$
\PV(X) [[t]] = \Omega^{0,\ast}(X, \wedge^\ast TX) [[t]] 
$$
that is, polyvector fields on $X$ with a formal parameter $t$ adjoined.    Our generalized BCOV theory is defined by an action functional $S$ on $\PV(X)[[t]]$ satisfying the classical master equation.  Our generalization of the classical BCOV theory is explained in section \ref{section classical bcov} 

\subsection{}
We then turn to the problem of quantization.   In section \ref{section quantization}, we give a general definition of a quantization of our generalized BCOV theory on a Calabi-Yau manifold, using the renormalization techniques developed in \cite{Cos11}.  Our notion of quantization relates to Givental's symplectic formalism for the study of the $A$-model \cite{Giv04}.  In section \ref{section givental}, we show that the quantum master equation in the quantum BCOV theory implies that the partition function is a state in a Fock space constructed from the cohomology of $X$.  

The first result of this paper -- which partially relies on some results from \cite{CosLi12} -- is the following. 
\begin{theorem*}
There exists a unique quantization of the BCOV theory on any elliptic curve $E$, subject to one additional constraint: the \emph{dilaton equation} (which is the mirror to the well-known dilaton equation in Gromov-Witten theory).
\end{theorem*}
The existence part of the proof will appear in \cite{CosLi12}, as part of our more general analysis. The uniqueness part is proved using obstruction theory.  The proof is presented in section \ref{section elliptic quantization} (although some detailed computations of Lie algebra cohomology are placed in an Appendix). 

\subsection{}
Our interpretation of BCOV theory has a close relationship to Givental's symplectic formalism \cite{Giv01,CoaGiv01,Giv04}, and to Barannikov's work \cite{Bar99,Bar00} on semi-infinite variations of Hodge structure.      In sections \ref{section givental} and \ref{section correlation} we show that the partition function for our version of BCOV theory provides a state in the Fock space built from $H^\ast(X)((t))$, equipped with the symplectic pairing defined by the formula
$$
\omega ( \alpha t^k, \beta t^l ) = \int_X \alpha \wedge \beta \op{Res} (t^{2-d} f(t) g(-t) \d t ).
$$
Here $d$ is the complex dimension of $X$, and $H^\ast(X)$ is the de Rham cohomology. 

In order to define the correlation functions of our theory, we need to choose a polarization of this symplectic vector space.   This is because the choice of a polarization allows us to identify the Fock space with a symmetric algebra, and the correlation functions are then defined as the Taylor coefficients of the partition function.  

 As we explain in section \ref{section correlation}, the Hodge filtration on $H^\ast(X)$ yields a Lagrangian subspace of this symplectic vector space, and so one half of a polarization.  In order to get a complete polarization, we need to choose a splitting of the Hodge filtration.   

 If $\br{F}$ denotes a filtration complementary to the Hodge filtration on $H^\ast(X)$, we define in section \ref{section correlation} correlation  functions
$$
\ip{-}_{g,n}^{X, \br{F} } : \left( H^\ast (X, \wedge^\ast T X ) [[t]]\right)^{\otimes n} \to \C.  
$$
Here $H^i(X,\wedge^j TX)$ denotes the coherent cohomology of $X$ with coefficients in the $j$'th exterior power of the tangent bundle of $X$.   Note that this algebra of polyvector fields on $X$ is a differential graded Frobenius algebra, with a trace given by the holomorphic volume form on $X$.

\subsection{}
Mirror symmetry predicts that the correlators of the BCOV theory on $X$ (with an appropriate choice of splitting of the Hodge filtration) coincide with the Gromov-Witten invariants of the mirror $X^\vee$.  
In a companion paper \cite{Li11} by the second-named author, this result is proved for the elliptic curve $E$.  
 
The $A$-model correlators (or Gromov-Witten invariants) of a projective variety $X^\vee$ are the linear maps
$$
\ip{-}_{g,n,d}^{X^\vee;A} : \left( H^\ast (X^\vee, \wedge^\ast T^\ast X^\vee) [[t]]\right)^{\otimes n} \to \C
$$
defined by by
$$
\ip{\alpha_1 t^{k_1}, \ldots, \alpha_n t^{k_n} }_{g,n,d}^{X^\vee; A} = \int_{[\mbar_{g,n,d}(X^\vee)]_{virt}} \left( \psi_1^{k_1} \op{ev}_1^\ast \alpha_1  \right) \cdots \left (\psi_n^{k_n} \op{ev}_n^\ast \alpha_n \right) 
$$
where $\mbar_{g,n,d}(X^\vee)$ is the Kontsevich moduli space of stable maps of degree $d$, $[\mbar_{g,n,d}(X^\vee)]_{virt}$ is the virtual fundamental class of this space, $\op{ev}_i : \mbar_{g,n,d}(X^\vee) \to X^\vee$ are the evaluation maps, and $\psi_i$ is the first Chern class of the cotangent line to the $i$'th marked point of $\mbar_{g,n,d}(X^\vee)$.

\subsection{}
In this paper, we prove that the correlators of the BCOV theory on an elliptic curve $E$ satisfy Virasoro constraints identical to those satisfied by the Gromov-Witten invariants of the dual elliptic curve $E^\vee$.

The Virasoro constraints for Gromov-Witten invariants were proposed in \cite{EguHorXio97} (and are still conjectural, in general).  In the case of a point, these Virasoro constraints are equivalent to Kontsevich's theorem \cite{Kon92} on intersection numbers on the moduli space of curves. 

The Virasoro in the case of an elliptic curve can be expressed in the following simple form:
$$
\ip{1 \cdot t^{l} , \alpha_1, \ldots, \alpha_n }_{g,d,n+1}^{E^\vee; A} 
=  \sum_{i = 1}^n \binom{l + \op{HW}(\alpha) }{l}  \ip{ \alpha_1 , \ldots,t^{l - 1} \alpha_i , \ldots, \alpha_n  }_{g,d,n}^{E^\vee; A}.
$$
Here $\alpha \in H^{\ast,\ast}(E)$ is an arbitrary element. The Hodge weight $\op{HW}(\alpha)$ is defined by $\op{HW}(\alpha) = p-1$ if $\alpha H^{p,q}(E)$. 
 
A second main theorem of this paper is as follows.
\begin{theorem*}
For any choice of splitting $\br{F}$ of the Hodge filtration on the cohomology of $E$, the $B$-model correlators satisfy the Virasoro constraints
$$
\ip{1 \cdot t^{l} , \alpha_1, \ldots, \alpha_n } _{g,n+1}^{E, \br{F}} 
=  \sum_{i = 1}^n \binom{l + \op{HW}(\alpha) }{l}  \ip{ \alpha_1 , \ldots,t^{l - 1} \alpha_i , \ldots, \alpha_n  } _{g,n}^{E, \br{F}} .
$$
Here, the Hodge weight of an element $\alpha \in \PV^{p,q}(E)$ is $\op{HW}(\alpha) = p -1$.  
\end{theorem*}
The proof, presented in section \ref{section virasoro}, is by obstruction theory. 

\subsection{}
The Virasoro constraints allow us to reduce the calculation of the even $B$-model correlators\footnote{Even means that all of the inputs are of even cohomological degree.} to the calculation of certain primitive correlators
$$
\ip{\omega t^{k_1}, \ldots, \omega t^{k_n}}_{g,n}^{E, \br{F}; B}
$$
where 
$$
\omega \in H^1 (E, TE)
$$
is the class with $\Tr (\omega) = 1$.   In \cite{OkoPan06}, the corresponding $A$-model correlators are called the \emph{stationary sector}.    

We next show that the correlators in the stationary sector are modular forms.  Let $E_\tau$ denote the elliptic curve with modulus $\tau \in \mbb{H}$.  For $\sigma \in \mbb{H}$, let $\br{F}_\sigma$ denote the splitting on the Hodge filtration on $H^1(E_\tau)$ arising from the complex conjugate to the Hodge filtration on $H^1(E_\sigma)$.  This makes sense, because there is a canonical isomorphism $H^1(E_\tau) \iso H^1(E_\sigma)$, arising from parallel transport along a path connecting $\tau$ and $\sigma$ in $\mbb{H}$.

We will let
$$\ip{-}_{g,n}^{E_\tau, \br{\sigma}; B} : \left( H^\ast(E_\tau, \wedge^\ast T E_\tau ) [[t]] \right)^{\otimes n} \to \C$$
denote the correlators for this splitting.  

One can check easily that the $\sigma \to i \infty$ limit of the space $\br{F}^1_{\sigma} \subset H^1(E_\tau)$ still defines a splitting of the Hodge filtration. We will let  $\ip{-}_{g,n}^{E_\tau, \infty; B} $ denote the correlators defined using this limiting splitting. 
\begin{theorem*}
The stationary correlators 
$$
\ip{\omega t^{k_1}, \ldots, \omega t^{k_n}}_{g,n}^{E, \br{\sigma}; B}
$$
are holomorphic functions of $\tau$ and anti-holomorphic functions of $\sigma$. Further, they are modular functions on $\mbb{H} \times \mbb{H}$ (using the diagonal $SL_2(\Z)$-action) of weight $(2g-2+2n,0)$.  In other words, if 
$$
A = \left( \begin{array}{c c} 
a & b\\
c & d
\end{array} \right) \in SL_2(\Z), 
$$
then
\begin{multline*}
\ip{1 \cdot t^{k_1}, \ldots, 1 \cdot t^{k_m}, \omega \cdot t^{l_1}, \ldots, \omega \cdot t^{l_n} }_{g,n+m}^{E_{A (\tau)}, \br{A(\sigma)} } \\ = (c \tau +d)^{2g-2+2n} \ip{1 \cdot t^{k_1}, \ldots, 1 \cdot t^{k_m}, \omega \cdot t^{l_1}, \ldots, \omega \cdot t^{l_n} }_{g,n+m}^{E_\tau, \br{\sigma} }.
\end{multline*}
\end{theorem*}

\subsection{}
Consider the generating function
$$
\ip{\omega t^{k_1}, \dots, \omega t^{k_n}}_{g,n}^{E^\vee, A} (q) = \sum q^d \ip{\omega t^{k_1}, \dots, \omega t^{k_n}}_{g,n,d}^{E^\vee, A} 
$$
of the $A$-model stationary correlation functions.  (Here $\omega \in H^2(E^\vee)$ is the fundamental class). Okounkov and Pandharipande show that this generating function is a quasi-modular form of weght $2g-2+n$, and they give an explicit formula for this form.

In \cite{Li11}, the corresponding $B$-model correlators are determined, and the following theorem is shown.  
\begin{theorem*}[Li]
The $A$-model and $B$-model stationary correlators coincide:
$$
\ip{\omega t^{k_1}, \dots, \omega t^{k_n}}_{g,n}^{E^\vee, A} (e^{2 \pi i \tau}) = \ip{\omega t^{k_1}, \dots, \omega t^{k_n}}_{g,n}^{E, \infty;  B}. 
$$
Here, the $B$-model correlators are defined with respect to the limiting splitting $\lim_{\sigma \to i \infty} \br{F}_\sigma$ of the Hodge filtration of $E$.
\end{theorem*}
Since the stationary correlators and the Virasoro constraints determine all even correlators, we see that all $A$- and $B$-model even correlators coincide.  

Li also shows that certain generalized Virasoro constraints hold on the $B$-model; the $A$-model version of these generalized Virasoro constraints were proved by Okounkov and Pandharipande \cite{OkoPan06a}.  These generalized Virasoro constraints allow one to reduce the computation of an arbitrary (not necessarily even) correlator to the stationary sector.  Thus, Li's results prove that \emph{all} correlators of the $A$ and $B$-model agree. 

A more precise statement of Li's theorem is as follows. Let 
$$
\Phi_\tau :   \oplus H^i(E^\vee, \wedge^j T^\ast E^\vee ) [-i-j] \to \oplus H^i(E_\tau, \wedge^j T E_\tau )[-i-j]
$$  
be the unique isomorphism of commutative bigraded algebras which is compatible with the trace on both sides.    Let us extend $\Phi_\tau$ to a $\C[[t]]$-linear isomorphism
$$
  H^\ast( E^\vee, \wedge^\ast T^\ast E)[[t]] \to H^\ast(E, \wedge^\ast TE) [[t]].
$$
\begin{ucorollary}[Li]
For all 
$$a_1, \ldots, a_n \in H^\ast(E^\vee, \wedge^\ast T^\ast E^\vee) [[t]],$$
the $A$-model correlators are the same as the $\br{\sigma} \to i \infty$ limit of the $B$-model correlators:
$$
\sum_d e^{2 \pi i d \tau} \ip{a_1, \ldots, a_n}^{E^\vee; A}_{g,n,d} = \ip{\Phi_\tau(a_1), \ldots, \Phi_\tau(a_n) }_{g,n}^{E_\tau, \infty ; B}.
$$
\end{ucorollary}
Thus, this theorem states that the generating function for the Gromov-Witten invariants of $E^\vee$, where degree $d$ invariants are accompanied by a coefficient of $e^{2 \pi i d \tau} = q^d$, coincides with the  $B$-model correlators of the elliptic curve $E_\tau$.  Note that the left hand side of this equality is \emph{a priori} only defined as a formal expansion in the parameter $q = e^{2 \pi i \tau}$.  The right hand side, however, is defined for all values of $\tau$ in the upper half-plane.  Thus, this equality should be understood as saying that the left hand side is an expansion of the right hand side as $\tau \to i \infty$ (or, equivalently, as $q \to 0$).

\tableofcontents

\section{Classical BCOV theory}
\label{section classical bcov}
In this section we will describe our generalization of the BCOV theory, at the classical level.  This classical theory is defined on a Calabi-Yau manifold of any dimension.

Thus, let $X$ be a Calabi-Yau manifold of dimension $d$. For simplicity we will assume that $X$ is compact, although this is not necessary.

Let 
$$
\PV^{i,j} (X)  = \Omega^{0,j} (X, \wedge^i T X ) 
$$
denote the space of polyvector fields on $X$.  $\PV(X)$ is a differential bi-graded algebra; the differential is the operator 
$$
\dbar : \PV^{i,j} (X) \rightarrow \PV^{i,j+1} (X).
$$
and the algebra structure arises from wedging polyvector fields.  The cohomological degree of an element of $\PV^{i,j}(X)$ is $i + j$. 

Give the algebra $\Omega^{\ast,\ast}(X)$ of forms on $X$ the $\dbar$ differential.  With this differential, $\Omega^{\ast,\ast}(X)$ is a differential bi-graded module for the differential bi-graded algebra $\PV(X)$.  The module structure is given by $\Omega^{0,\ast}(X)$ linear map
\begin{align*}
\PV^{k,l} (X) \otimes \Omega^{i,j} ( X)  & \to \Omega^{i - k, j + l } ( X) \\
\alpha \otimes \phi &\mapsto \alpha \vee \phi.
\end{align*}
This map is the unique $\Omega^{0,\ast}(X)$ linear extension of the contraction map
$$
\PV^{k,0}(X) \otimes \Omega^{i,0}(X) \to \Omega^{i-k,0}(X).
$$

Let 
$$
\omega \in \Omega^{d,0} ( X )
$$
denote the holomorphic volume form on $X$.   

Contracting with $\omega$ yields an isomorphism
\begin{align*}
\alpha & \mapsto \alpha \vee \omega \\
\PV^{i,j}(X) & \iso \Omega^{d-i,j}(X).
\end{align*}
Define an operator 
$$
\partial : \PV^{i,j}(X) \to \PV^{i-1,j}(X)
$$
by the formula
$$
(\partial \alpha) \vee \omega =  \partial ( \alpha \vee \omega).
$$
Thus, $\partial$ is the operator on $\PV^{\ast,\ast}(X)$ which corresponds to the holomorphic de Rham differential $\partial$ on $\Omega^{\ast,\ast}(X)$ under the isomorphism $\PV^{i,j}(X) \iso \Omega^{d-i,j}(X)$. 

\subsection{}
The operator $\partial$ endows the space $\PV(X)$ of polyvector fields on $X$ of with the structure of a bigraded Batalin-Vilkovisky algebra.   This means that $\partial$ is an order two differential operator with respect to the algebra structure on $\PV(X)$.

If $\alpha \in \PV(X)$, let 
$$\alpha_l : \PV(X) \to \PV(X)$$
denote the operator of left multiplication with $\alpha$.  The statement that $\partial$ is an order two differential operator means that, for all $\alpha,\beta,\gamma$, 
$$
[[[ \partial, \alpha_l ] , \beta_l ], \gamma_l ]   = 0.
$$

It follows from this that there is a uniquely defined bilinear map
$$
\{-,-\} : \PV^{i,j}(X) \otimes \PV^{k,l}(X) \to \PV^{ i + k - 1, j + l } ( X) 
$$
which satisfies
$$
[[ \partial, \alpha_l ], \beta_l ]  \gamma = \{\alpha,\beta\} \gamma
$$
for all $\alpha,\beta,\gamma \in \PV(X)$. 

It follows from the fact that $\partial^2 = 0$ that $\{-,-\}$ is a Gerstenhaber bracket on the bi-graded algebra $\PV(X)$ (see \cite{Get94} for details).

\subsection{}
Define a map
$$
\Tr : \PV(X) \to \C
$$
by
$$
\Tr (\alpha ) =  \begin{cases} 
 \int_X  ( \alpha \vee \omega) \wedge \omega  & \text{ if } \alpha \in \PV^{d,d}(X) \\
0 & \text{ if } \alpha \not\in \PV^{d,d}(X).
\end{cases}
$$
The pairing
\begin{align*}
\PV^{i,j}(X) \otimes \PV^{d-i,d-j}(X) & \to \C \\ 
\alpha \otimes \beta & \to \Tr ( \alpha \beta) 
\end{align*}
is non-degenerate (that is, it has no kernel). 
\begin{lemma}
The operator $\dbar$ is skew self adjoint for the pairing $\Tr (\alpha \beta)$, and the operator $\partial$ is self adjoint for this pairing.
\end{lemma}

\subsection{}
The space of fields for the original BCOV theory is the subspace
$$
\mc{H} = \left(\op{Ker} \partial \right)[2]\subset \PV(X) [2]. 
$$
(The shift of two puts the deformations of complex structure of $X$ -- which lie in $\Omega^{0,1}(X, TX)$ -- in degree $0$). 

The original BCOV action on $\mc{H}$ is defined by the formula 
$$
S(a) = \frac{1}{2}Tr (a \dbar \partial^{-1} a ) + \frac{1}{6} \Tr (a^3).
$$

\subsection{}
There are many difficulties with the BCOV action, as originally defined.  The most obvious difficulty is that it involves the expression $\partial^{-1}$.  This difficulty can be resolved by restricting attention to those fields $a$ in $\Im \partial \subset \Ker \partial$.

The second difficulty is that the space of fields is $\Ker \partial$.   Note that $\Ker \partial$ is the global sections of a sheaf of cochain complexes on $X$, whose sections on an open subset $U \subset X$ is simply
$$
\Ker \partial \mid_U \subset \PV(U).
$$
Note that this sheaf of cochain complexes is \emph{not} fine: the hypercohomology with coefficients in this sheaf does not coincide with the cohomology of $Ker \partial \subset \PV(X)$.

In this paper we will consider a modification of the BCOV theory which works on a Calabi-Yau manifold of any dimension.  In this modified theory, the space of fields is the \emph{derived} version of the kernel of $\partial$.

\subsection{}
The operator
$$
\partial : \PV(X) \to \PV(X)
$$
is a cochain map of cohomological degree $-1$.  Thus, $\partial$ can be viewed as an action of the Abelian Lie algebra $\C[1]$ situated in degree $-1$, or, equivalently, as an action of the commutative algebra $\C[\eps]$ where $\eps$ is placed in degree $-1$.  

The derived version of the kernel of $\partial$ is the homotopy fixed points for the action of $\C[1]$.  Abstractly, we can define this as the derived tensor product
$$
\PV(X) \otimes^{\mbb L}_{\C[\eps]} \C.
$$
A concrete model for this complex is the complex
$$
\PV(X)[[t]]
$$
with differential $\dbar - t \partial$.   We will take this complex to be the space of fields for our modified BCOV theory (after, as before, shifting by $2$).  Thus, let
$$
\Fields(X) = \PV(X)[[t]] [2],
$$
and let 
$$
Q = \dbar - t \partial : \Fields(X) \to \Fields(X)
$$
be the differential on $\Fields(X)$.  Note that $\Fields(X)$ is a topological vector space, and in fact a nuclear Fr\'{e}chet space.  

\subsection{}
We are interested constructing a local action functional on $\Fields(X)$ which satisfies the classical master equation.  Before we do this, we need to introduce some notation for various classes of functions on $\Fields(X)$.  

Let $\br{\PV}(X)$ denote the space of distributional polyvector fields on $X$.  Thus, if $\mc{D}(X)$ denotes the space of distributions on $X$,
$$
\br{\PV}(X) = \PV(X) \otimes_{\cinfty(X)} \mc{D}(X).
$$
Let 
$$
\br{\Fields}(X) = \br{\PV}(X)[[t]][2].
$$
We will give $\br{\Fields}(X)$ the differential $Q = \dbar - t \partial$, as before.

Note that both $\Fields(X)$ and $\br{\Fields}(X)$ are cochain complexes of topological vector space in a natural way; in fact, both are complete nuclear spaces.   The category of complete nuclear spaces is a symmetric monoidal category, with the completed projective tensor product. Further, the category of nuclear spaces is closed under all products and under countable coproducts.  If $V$ is a nuclear space, we will let  
$$
\what{\Sym}^\ast V = \prod_i \Sym^i (V) 
$$
denote the completed symmetric algebra on $V$, where $\Sym^i V$ is defined using the completed projective tensor product.

Now suppose that $V$ is a nuclear space with the property that the strong dual $V^\ast$ of $V$ is also a nuclear space. (The spaces $\Fields(X)$ and $\br{\Fields}(X)$ both have this property).  Then we define the algebra of formal power series on $V$ by
$$
\Oo(V) = \what{\Sym}^\ast (V^\ast).
$$ 
This definition extends, in a straightfoward way, to cochain complexes of nuclear spaces.  We are particularly interested in the spaces $\Oo(\Fields(X))$ and $\Oo(\br{\Fields}(X))$ of functions on spaces of fields on $X$. 

We will let 
$$
Q : \Oo(\Fields(X)) \to \Oo(\Fields(X))
$$ 
denote the differential induced from the differential $Q$ on $\Fields(X)$.

The pairing
$$
\ip{a,b} = \op{Tr} (a b ) 
$$
of cohomological degree $-2d $ on $\PV(X)$ leads to an isomorphism
$$
\PV(X)^\ast = \br{\PV}(X) [2d]
$$
and so to an isomorphism
$$
\Fields(X)^\ast = t^{-1} \br{\PV}(X)[t^{-1}] [2d -4] .
$$
Here, $t$ is given degree $2$ as usual, and we view $t^{-1} \C[t^{-1}]$ as the dual to $\C[[t]]$ via the residue pairing. 

Thus,
$$
\Oo(\Fields(X)) = \what{\Sym}^\ast \left( t^{-1} \br{\PV}(X)[t^{-1}] [2d -4] \right).
$$ 
Similarly, 
$$
\Oo(\br{\Fields}(X)) = \what{\Sym}^\ast \left( t^{-1} \PV(X)[t^{-1}] [2d -4] \right).
$$

\subsection{}
If $f \in \Fields(X)$, then there is a derivation 
$$
\frac{\partial}{\partial f} : \Oo(\Fields(X)) \to \Oo(\Fields(X)).
$$
This is the continuous unique derivation which, on the generators $\Fields(X)^\ast$ of $\Oo(\Fields(X))$, is given by pairing with $f$.  Thus, if $f$ has cohomological degree $i$, the derivation $\frac{\partial}{\partial f}$ also has cohomological degree $i$.

If $\Phi \in \Oo(\Fields(X))$, and then we can define the Taylor coefficients of $\Phi$ as follows.  These are continuous linear maps
$$
D_n \Phi : \Fields(X)^{\otimes n} \to \C
$$
which are uniquely defined by the property that
$$
D_n \Phi ( f_1 \otimes \cdots \otimes f_n) = \left( \frac{\partial}{\partial f_1} \cdots \frac{\partial}{\partial f_n} \Phi \right)(0).
$$
One has
$$
\Phi(f) = \sum \frac{1}{n!} D_n\Phi (f^{\otimes n}).
$$
(once both sides are suitable interpreted, using, for example, the functor of points formalism \cite{DelMor99}).

\subsection{}
We are particularly interested in a subspace 
$$\Ool(\Fields(X)) \subset \Oo(\Fields(X))$$
of \emph{local functionals}.  
\begin{definition}
A functional $\Phi \in \Oo(\Fields(X))$ is \emph{local} if the Taylor coefficients
$$
D_n \Phi : \Fields(X)^{\otimes n} \to \C,
$$
which are distributions on $X^n$ with coefficients in a certain vector bundle, have the following properties.
\begin{enumerate}
\item $D_n \Phi$ is supported in the small diagonal $X \subset X^n$.
\item The wave-front set (or microsupport) of $D_n \Phi$ is the conormal bundle to the small diagonal in $X$. 
\end{enumerate}
\end{definition}
The second condition is equivalent to the statement that we can find differential operators
$$
A_1, \ldots, A_n : \Fields(X) \to \cinfty(X)
$$
such that
$$
D_n (\Phi_1 \otimes \cdots \otimes \Phi_n) = \int_X A_1 (\Phi_1) \cdots A_n (\Phi_n) \d\op{Vol}
$$
for some volume form $\d\op{Vol}$ on $X$. 

The space $\Ool(\Fields(X))$ is a subcomplex of $\Oo(\Fields(X))$, when the latter is equipped with the differnetial $Q = \dbar - t \partial$. 

\subsection{}
We want to construct a function $I \in \Ool(\Fields(X))$ which satisfies the \emph{classical master equation}.  The classical master equation is the Maurer-Cartan equation for a certain Lie bracket on $\Ool(\Fields(X))$.

We will first describe a Lie bracket on the subalgebra
$$\Oo(\br{\Fields}(X)) \subset \Oo(\Fields(X)).$$ This is the subalgebra consisting of all funcitonals $\Phi$ whose Taylor coefficients, which are \emph{a priori} distributional sections of a vector bundle on $X$, are in fact smooth sections.

The Lie bracket on $\Oo(\br{\Fields}(X))$ will be a Poisson bracket of cohomological degree $2d - 5$. We will define it on the generators of $\Oo(\br{\Fields}(X))$; the Leibniz rule allows one to extend it to all elements of $\Oo(\br{\Fields}(X))$.

Recall that
$$
\br{\Fields}(X)^\ast = t^{-1} \PV(X) [t^{-1} ] [2 d - 4 ].
$$
The Poisson bracket on $\br{\Fields}(X)^\ast$ is defined by the pairing
$$
f(t) \alpha \otimes g(t) \beta \to \frac{1}{2 \pi i } \Res_{t = 0} t f(t) g(-t) \op{Tr} (\partial \alpha) \beta.  
$$
Note that this is a symmetric map of cohomological degree $2 d - 5$, and so induces a Poisson bracket on $\Oo(\br{\Fields}(X))$ of degree $2 d - 5$.  

Although this Poisson bracket does not extend to $\Oo(\Fields(X))$, it is clear that it extends to a bilinear map\footnote{This bilinear map is only separetely continuous, and so does not extend to the completed projective tensor product.}
$$
\Oo(\Fields(X)) \times \Oo(\br{\Fields}(X)) \to \Oo(\Fields(X)).
$$
In particular, any $S \in \Ool(\Fields(X))$ defines a continuous linear map
$$
\{S , - \} : \Oo(\br{\Fields}(X)) \to \Oo(\Fields(X)).
$$ 
\begin{lemma}
For all $S \in \Ool(\Fields(X))$, the map $\{S,-\}$ extends uniquely to a continuous map
$$
\Oo(\Fields(X)) \to \Oo(\Fields(X)).
$$
Further, $T \in \Ool(\Fields(X)) \subset \Oo(\Fields(X))$, then $\{S,T\} \in \Ool(\Fields(X))$. 
\end{lemma}
\begin{proof}
See \cite{Cos11}, Chapter 5.
\end{proof}

\begin{definition}
A function $S \in \Ool(\Fields(X))$ satisfies the \emph{classical master equation} if
$$
Q S + \tfrac{1}{2} \{S,S\} = 0.
$$
\end{definition}
\subsection{}
Now we are ready to define the classical action functional for our generalized BCOV theory. 
\begin{definition}
Define the classical BCOV action functional $I \in \Ool(\Fields(X))$ by saying that the Taylor coefficients
$$
D_n I : \Fields(X)^{\otimes n} \to \C
$$
are defined by
$$
D_n I ( t^{k_1} \alpha_1, \ldots, t^{k_n} \alpha_n) = \begin{dcases}
\int_{\mbar_{0,n}} \psi_1^{k_1} \cdots \psi_n^{k_n} \Tr ( \alpha_1 \cdots \alpha_n) & \text{ if } n \ge 3 \\
0 & \text{ if } n < 3.
\end{dcases}
$$
\end{definition}

\begin{lemma}
$I$ satisfies the classical master equation
$$
Q I + \tfrac{1}{2} \{I,I\} = 0. 
$$
\end{lemma}
\begin{proof} 
In what follows, we will use the notation 
$$
\ip{ \tau_{k_1} \dots \tau_{k_n} } = \int_{\mbar_{0,n}} \psi_1^{k_1} \dots \psi_n^{k_n}. 
$$
 Note that $\dbar I=0$. Thus, we only need to consider the term $- t \partial$ in the differential $Q$. We have
\begin{eqnarray*}
     &&(QI)(t^{k_1}\alpha_1, \cdots, t^{k_n}\alpha_n)\\
   &=&-\sum_{i} \pm \ip{\tau_{k_1}\cdots\tau_{k_i+1}\cdots \tau_{k_n}}_0
     \Tr \alpha_1\cdots\partial\alpha_i\cdots\alpha_n\\
     &=&{1\over 2}\sum_{i} \pm \ip{\tau_{k_1}\cdots\tau_{k_i+1}\cdots \tau_{k_n}}_0
     \Tr \{\alpha_i,\alpha_1\cdots\hat \alpha_i\cdots\alpha_n\}\\
&=&{1\over 2} \sum_{i\neq j}\pm \ip{\tau_{k_1}\cdots\tau_{k_i+1}\cdots
    \tau_{k_n}}_0 \Tr
    \{\alpha_i,\alpha_j\}\alpha_1\cdots\hat{\alpha_i}\cdots\hat{\alpha_j}\cdots\alpha_n
\end{eqnarray*}
where we have used the formula
$$
   \Tr (\pa \alpha)\beta=-{1\over2} \Tr\{\alpha, \beta\}
$$
which follows from the BV relation $\pa(\alpha\beta)=(\pa
\alpha)\beta+(-1)^{|\alpha|}\alpha \pa \beta+\fbracket{
\alpha,\beta}$ and the self-adjointness of $\pa$ with respect to the
trace pairing. On the other hand,
\begin{eqnarray*}
     &&   \{I,I\} (t^{k_1}\alpha_1, \cdots,
t^{k_n}\alpha_n)\\
     &=&\sum_{J\subset \{1,\cdots, n\}}\pm \langle \tau_0 \prod_{i\in J}\tau_{k_i} \rangle_0
\langle \tau_0 \prod_{j\in J^c}\tau_{k_j} \rangle_0 \Tr
\bracket{\partial \prod_{i\in J}\alpha_i}\prod_{j\in J^c}\alpha_j\\
&=&-{1\over 2}\sum_{J\subset \{1,\cdots, n\}}\pm \langle \tau_0
\prod_{i\in J}\tau_{k_i} \rangle_0 \langle \tau_0 \prod_{j\in
J^c}\tau_{k_j} \rangle_0 \Tr \fbracket{ \prod_{i\in J}\alpha_i,
\prod_{j\in J^c}\alpha_j}\\ &=&-{1\over 2}\sum_{J\subset \{1,\cdots,
n\}}\pm \langle \tau_0 \prod_{i\in J}\tau_{k_i} \rangle_0 \langle
\tau_0 \prod_{j\in J^c}\tau_{k_j} \rangle_0 \sum_{i\in J, j\in J^c}
\Tr \{\alpha_i,
\alpha_j\}\alpha_1\cdots\hat{\alpha_i}\cdots\hat{\alpha_j}\cdots\alpha_n
\\ &=& -{1\over 2(n-2)}\sum_{J\subset \{1,\cdots, n\}}\pm \langle
\tau_0 \prod_{i\in J}\tau_{k_i} \rangle_0 \langle \tau_0 \prod_{j\in
J^c}\tau_{k_j} \rangle_0 \sum_{i\in J, j\in J^c} \sum_{k\neq i,j}\Tr
\{\alpha_i,
\alpha_j\}\alpha_1\cdots\hat{\alpha_i}\cdots\hat{\alpha_j}\cdots\alpha_n\\
&=&-{1\over 2(n-2)}\sum_{i,j,k}\sum_{\{i,k\}\subset J, j\in J^c}\pm
\langle \tau_0 \prod_{i\in J}\tau_{k_i} \rangle_0 \langle \tau_0
\prod_{j\in J^c}\tau_{k_j} \rangle_0 \Tr \{\alpha_i,
\alpha_j\}\alpha_1\cdots\hat{\alpha_i}\cdots\hat{\alpha_j}\cdots\alpha_n\\
&&-{1\over 2(n-2)}\sum_{i,j,k}\sum_{i\in J, \{j,k\}\subset J^c}\pm
\langle \tau_0 \prod_{i\in J}\tau_{k_i} \rangle_0 \langle \tau_0
\prod_{j\in J^c}\tau_{k_j} \rangle_0 \Tr \{\alpha_i,
\alpha_j\}\alpha_1\cdots\hat{\alpha_i}\cdots\hat{\alpha_j}\cdots\alpha_n
\\
&=&-{1\over 2(n-2)}\sum_{i,j,k} \pm\langle \tau_{k_1}\cdots
\tau_{k_j+1}\cdots \tau_{k_n} \rangle_0  \Tr
    \{\alpha_i,\alpha_j\}\alpha_1\cdots\hat{\alpha_i}\cdots\hat{\alpha_j}\cdots\alpha_n\\&&-{1\over 2(n-2)}\sum_{i,j,k}\pm\langle \tau_{k_1}\cdots \tau_{k_i+1}\cdots \tau_{k_n} \rangle_0  \Tr
    \{\alpha_i,\alpha_j\}\alpha_1\cdots\hat{\alpha_i}\cdots\hat{\alpha_j}\cdots\alpha_n\\
&=&- \sum_{i,j}\pm \ip{\tau_{k_1}\cdots\tau_{k_i+1}\cdots
    \tau_{k_n}}_0 \Tr
    \{\alpha_i,\alpha_j\}\alpha_1\cdots\hat{\alpha_i}\cdots\hat{\alpha_j}\cdots\alpha_n
\end{eqnarray*}
where we have used the topological recursion relation
$$
   \ip{\tau_{k_1+1}\tau_{k_2}\cdots \tau_{k_n}}_0=\sum_{1\in I, \{2,3\}\subset I^c} \ip{\tau_0 \prod_{i\in I}\tau_{k_i}}_0\ip{\tau_0\prod_{j\in I^c}\tau_{k_j}}
$$
The classical master equation now follows.
\end{proof}
\begin{remark}
We will give another construction of the classical action functional (and another proof that the master equation holds) in section \ref{section_cone}.
\end{remark}

\section{Quantization}

\label{section quantization}
In this section we will explain what it means to quantize the BCOV theory on a Calabi-Yau $X$.

We will use the formalism for renormalization of quantum field theories developed in \cite{Cos11}.  This involves the choice of an additional datum: a \emph{gauge fixing operator} for the theory.  However, as explained in \cite{Cos11}, the notion of quantization is independent, up to homotopy, of the gauge fixing operator chosen.   More precisely, it is shown in \cite{Cos11} that there is a simplicial set of quantization, which is a fibration over the simplicial set of gauge fixing conditions.  
Further, the simplicial set of gauge fixing conditions is normally contractible (and is contractible in the example under consideration).  Since the simplicial set of quantizations fibres over that of gauge fixing conditions, any two fibres are homotopy equivalent, in a way canonical up to all higher homotopies. 

In what follows, we will first discuss the notion of quantization with a fixed  gauge fixing condition, arising from a K\"ahler metric on our Calabi-Yau $X$.  Later we will explain how to relate quantizations for different gauge fixing conditions. 

\subsection{}
Thus, let us choose a K\"ahler metric on $X$. As before, let $\PV(X)$ denote the polyvector fields on $X$.  The choice of metric leads to an operator
$$
\dbar^\ast : \PV^{i,j}(X) \to \PV^{i-1,j}(X).
$$
Let
$$
K_t \in \oplus_{i,j} \PV^{i,j}(X) \otimes \PV^{d-i,d-j}(X)
$$
denote the heat kernel for the operator $[\dbar,\dbar^\ast]$. 

As before, let
$$
\Fields(X) = \PV(X)[[t]] [2]. 
$$
We will consider $\PV(X)[2]$ as the subspace $t^0 \PV(X)[2]$ of $\Fields(X)$.  In this way, the heat kernel $K_t$ can be viewed as an element of $\Fields(X) \otimes \Fields(X)$.

We will let
$$
\Delta_L : \Oo(\Fields(X)) \to \Oo(\Fields(X))
$$
be the operator of contracting with 
$$(\partial \otimes 1) K_L.$$
Thus, $\Delta_L$ is the unique continuous order $2$ differential operator with the property that its restriction to 
$$
\Sym^2 (\Fields(X)^\ast) \subset \Oo(\Fields(X))
$$
is given by pairing with $(\partial \otimes 1) K_L$.  This operator is of cohomological degree $2d-5$.  

If $\Phi, \Psi \in \Oo(\Fields)$, let
$$
\{\Phi,\Psi\}_L = \Delta_L ( \Phi \Psi ) -  \Delta_L ( \Phi) \Psi - (-1)^{\abs{\Phi}} \Phi \Delta_L \Psi.
$$
This defines a Poisson bracket of cohomological degree $2d - 5$.

\subsection{}
We will let
$$
P(\eps,L) = \int_\eps^L \left( \dbar^\ast \partial \otimes 1 \right) K_u \d u \in \Fields \otimes \Fields.
$$
If we set $\eps = 0$ and $L = \infty$, then we find the propagator $P(0,\infty)$ for the BCOV theory.

Let
$$
Q : \Oo(\Fields(X)) \to \Oo(\Fields(X)) 
$$
denote the differential which arises from the differential $\dbar - t \partial$ on $\Fields(X)$.

Let
$$
\partial_{P(\eps,L)} : \Oo(\Fields(X)) \to \Oo(\Fields(X)) 
$$
be the operator corresponding to contracting with $P(\eps,L)$.

We have
\begin{equation}
[Q, \partial_{P(\eps,L)} ] = \Delta_\eps - \Delta_L \label{eqn_dagger} \tag{$\dagger$}
\end{equation}
Thus, $\partial_{P(\eps,L)}$ is a homotopy between the operators $\Delta_L$ and $\Delta_\eps$. 
\subsection{}
Let 
$$\Oo^+(\Fields(X)) [[\hbar]] \subset \Oo(\Fields(X))[[\hbar]]$$
be the subspace consisting of elements which are at least cubic.  $\Phi \in \Oo^+(\Fields(X)) [[\hbar]]$,  Let
$$
\EA{P(\eps,L), \Phi} = \hbar \log \exp \left ( \hbar \partial_{P(\eps,L)} \right) \exp \left( \Phi / \hbar \right) \in \Oo^+(\Fields(X)) [[\hbar]].
$$
Sending 
$$
\Phi \mapsto \EA{P(\eps,L), \Phi}
$$
is the renormalization group flow from scale $\eps$ to scale $L$.  This is discussed in much greater detail in \cite{Cos11}.

\subsection{}
Now we have the notation necessary to give a definition of a \emph{quantization} of the BCOV theory on a Calabi-Yau manifold $X$.

\begin{definition}
A quantization of the BCOV theory on $X$ is given by a K\"ahler metric $g$ on $X$, and a collection of functionals
$$
\F[L] = \sum \hbar^g \F_g[L] \in \Oo (\Fields(X))[[\hbar]]
$$
for each $L \in \R_{> 0}$, with the following properties.
\begin{enumerate}
\item  The renormalization group flow equation. This says that 
$$
\F[L] = \EA{P(\eps,L), \F[\eps]}
$$
for all $L$ and $\eps$. 
\item The quantum master equation. This says that
$$
Q \F[L] + \hbar \Delta_L \F[L] + \frac{1}{2} \{\F[L],\F[L]\}_L = 0,
$$
or, equivalently,
$$
(Q + \hbar \Delta_L ) \exp ( \F[L] / \hbar ) = 0.
$$
The homotopy property (\ref{eqn_dagger}) implies that, if $\F[L]$ satisfies the renormalization group equation, then the QME at scale $L$ is equivalent to the QME at scale $\eps$. 
\item The locality axiom, as in \cite{Cos11}. This says that $\F[L]$ has a small $L$ asymptotic expansion in terms of local functionals. 
\item The $L \to 0$ limit of $\F_0[L]$ is the BCOV interaction $I \in \Ool(\Fields(X))$. 
\item The function $\F_g$ is of cohomological degree
$$
(\dim X - 3)(2g-2) .
$$
(This condition corresponds to the fact that, in the $A$ model, the moduli space of maps $\br{\mscr{M}}_{g,n}(X)$ is of dimension $( 3 - \op{dim} X  ) ( 2g - 2 ) + 2 n$.   The $2n$ term does not appear in our expression because we shifted the grading of $\PV(X)$). 
\item
We will give $\Fields(X) = \PV(X)[[t]] [2]$ an additional grading, which we call Hodge grading,  by saying that something in 
$$t^m \Omega^{0,\ast} (\wedge^k T X) = \PV^{k,\ast}(X)$$ has Hodge weight $k + m - 1$.    We will let $HW(\alpha)$ denote the Hodge weight of an element $\alpha \in \E$.

Then, the functions $\F_g$ must be of Hodge weight
$$
(\op{dim} X - 3) (g - 1) .
$$
(This condition matches up with the $A$ model fact that the fundamental class of $\br{\mscr{M}}_{g,n}(X)$ is of Hodge type 
$$((3 - \op{dim} X ) (g - 1) + n,(3 - \op{dim} X ) (g - 1) + n).$$ 
\end{enumerate}
\end{definition}

\subsection{}
Note that one can consider families of quantizations over any differential graded base ring.  Indeed, if $A$ is a such a dg base ring, then we can consider our space of functionals to be $\Oo(E) \otimes A [[\hbar]]$, and we can modify our differential $Q$ to be $\dbar - t \partial + \d_A$, where $\d_A$ is the differential on $A$. 

The main theorems of \cite{Cos11} apply when we work over a certain class of differential graded commutative Fr\'echet algebras.  These Fr\'echet algebras are associated to data $(Z,A,I,\d)$ where $Z$ is an auxiliary manifold with corners, $A$ is a finite rank bundle of graded commutative on $Z$, $I \subset A$ is a bundle of nilpotent maximal ideals whose quotient is the trivial bundle, and $\d : \Gamma(Z,A) \to \Gamma(Z,A)$ is a derivation.   This data yields a differential graded Fr\'echet algebra $\A = \Gamma(Z,A)$.      

By considering suitable base rings, one defines a notion of homotopy of quantizations of the BCOV theory.
\begin{definition}
Let $\{\F[L], g\}$, $\{\F'[L], g\}$ be two different quantizations defined using the same metric $g$. 
A homotopy between two quantizations $\{\F[L], g\}$, $\{\F'[L],g\}$ is a family of quantizations over the differential graded base ring $\Omega^\ast([0,1])$, satisfying the obvious analogues of the axioms listed above, and which restricts to $\{\F[L]\}$ at $0$ and to $\{\F'[L]\}$ at $1$.  
\end{definition} 
More generally, one can consider families over $\Omega^\ast ( \Delta^n)$, where $\Delta^n$ is the $n$-simplex.  In this way we see that the set of quantizations of the BCOV theory has a natural enlargement to a simplicial set. 

\subsection{}
So far we have been considering quantizations with a fixed K\"ahler metric $g$ on $X$.  We can also consider homotopies between quantizations which have different metric. 

Thus, let $g_0,g_1$ be K\"ahler metrics on $X$ and let $\{\F[L]_0, g_0\}$, $\{\F[L]_1, g_1\}$ be quantizations defined with respect to the metric $g_0$ and $g_1$.  Let $g_s$ for $s \in [0,1]$ be a smooth family of metrics connecting $g_0$ to $g_1$.  Then, as explained in \cite{Cos11}, Chapter 5, one can define, in this setting, a family of heat kernels and propagators in $\E \otimes \E \otimes \Omega^\ast([0,1])$, which use the family $g_s$ of metrics as a gauge fixing conditions.   The heat kernel $K_{L, s, \d s}$ for this family is defined to be the heat kernel for the operator $[\dbar+ \d_{dR}, \dbar^\ast_s]$.  Here $\d_{dR}$ is the de Rham differential on $\Omega^\ast([0,1])$, and 
$$
\dbar^\ast_s : \E (X) \to \E(X) \otimes \cinfty( [0,1]_s ) 
$$
is the family version of the $\dbar^\ast$ operator for the family of metrics $\{\g_s \mid s \in [0,1]\}$. 

This allows one to define the notion of quantization for the family of gauge fixing conditions $\dbar^\ast_s$.   Such a quantization is given by a family of functions
$$
\F[L] \in \Oo (\E(X)) \otimes \Omega^\ast ( [0,1]_s ) [[\hbar]],
$$
as before, which satisfy all of the axioms listed earlier, except using the family heat kernel $K_{L, s, \d s}$ in place of the usual heat kernel $K_{L}$. 
\begin{definition}
A homotopy between quantizations $\{\F[L]_0, g_0\}$ and $\{\F[L]_1, g_1\}$ is a smooth family $g_t$ of metrics connecting $g_0$ and $g_1$, together with a quantization over $\Omega^\ast([0,1])$.  
\end{definition}
More generally, this construction allows one to define a simplicial set of quantizations, which we denote $\op{Quant}$.  The $0$-simplices are quantizations $\{\F[L], g\})$ where $g$ is a K\"ahler metric on $X$ and $\{\F[L]\}$ is a quantization defined using this metric.  The higher simplices are given by families of metrics parametrized by $\Delta^n$ and a corresponding family of quantizations.  

Let $\op{Met}(X)$ denote the simplicial set whose $n$-simplices are smooth families of K\"ahler metrics on $X$.  In \cite{Cos11}, Chapter 5, the following theorem is shown.
\begin{theorem}
The map $\op{Quant} \to \op{Met}(X)$ is a fibration of simplicial sets.
\end{theorem}
Since $\op{Met}(X)$ is contractible, this theorem tells us that any two fibres of this map are canonically homotopy equivalent.  Thus, we do not lose any generality if we fix any K\"ahler metric $g$ on $X$ and only consider quantizations using this metric.    In what follows, we will normally do this. 

\subsection{}
So far, we have defined the notion of quantization of the classical BCOV theory.    The axioms we have listed are the fundamental axioms of the quantum field theory formalism of \cite{Cos11} -- namely, the renormalization group equation and the quantum master equation -- as well as axioms corresponding to the dimension axioms of Gromov-Witten theory.  

However, the Gromov-Witten invariants of an algebraic variety have several other important axioms which we would like to consider in the $B$-model.  Two particularly important axioms are the \emph{string axiom} and the \emph{dilaton axiom}.   Let us recall the statement of these axioms.

Let $Y$ be a projective algebraic variety, and let $\mbar_{g,n,\beta} (Y)$ denote the Kontsevich moduli space of stable maps from curves of genus $g$ to $Y$, whose fundamental class is $\beta \in H_2(Y)$.  Let 
$$
[\mbar_{g,n,\beta} (Y)]^{vir} \in H_{(1-g)(\op{dim} Y - 3) + 2n + \int_\beta c_1(Y) } (\mbar_{g,n,\beta} (Y))
$$
denote the virtual fundamental class of this moduli space. 

Then, the Gromov-Witten invariants of $Y$ are defined by
$$
\ip{t^{k_1} \alpha_1, \dots, t^{k_n} \alpha_n   }^{GW}_{g,n,\beta} = \int_{[\mbar_{g,n,\beta} (Y)]^{vir}} \psi_1^{k_1} \op{ev}_1^\ast \alpha_1 \dots \psi_n^{k_n} \op{ev}_n^\ast \alpha_n
$$
where $\alpha_i \in H^\ast(Y)$,  $\op{ev}_i : \mbar_{g,n,\beta} (Y) \to Y$ are the evaluation maps, and $\psi_i \in H^2 (\mbar_{g,n,\beta} (Y) )$ is the pull-back of the $i$'th $\psi$ class on the usual Deligne-Mumford moduli space $\mbar_{g,n}$ of curves. 

The \emph{dilaton axiom} states that
$$
\ip{t \cdot 1, t^{k_1} \alpha_1, \dots, t^{k_n} \alpha_n   }^{GW}_{g,n+1,\beta} = (2g - 2 + n )\ip{t^{k_1} \alpha_1, \dots, t^{k_n} \alpha_n   }_{g,n,\beta}  
$$
where $1 \in H^0 ( Y)$ is the identity in the cohomology ring of $Y$.  

The \emph{string axiom} states that
$$
\ip{1, t^{k_1} \alpha_1, \dots, t^{k_n} \alpha_n   }^{GW}_{g,n+1,\beta} = \sum_{i = 1}^n \ip{t^{k_1} \alpha_1, \dots, t^{k_i - 1} \alpha_i , \dots, t^{k_n} \alpha_n   }_{g,n,\beta}  
$$
Thus, the string and dilaton axioms allow one to remove a single $t \cdot 1$ or $1$ from the correlators.  

\subsection{}
We would like to state the string and dilaton axioms for the $B$-model.   However, in the $B$-model, these axioms will only hold up to homotopy.  We will start by describing the dilaton axiom, as the string axiom in the $B$-model is more difficult to describe.    In order to state these axioms precisely, we need some lemmas.

Let $1 \in \E(X)$ be the element corresponding to the function $1$ in the ring $\PV^{0,0}(X) = \cinfty(X)$.     Similarly, for any $k$, we have an element $t^k 1 \in \E(X)$.  Thus, we get differential operators
$$
\partial_{t^k 1} = \frac{\partial}{\partial\left( t^k 1 \right) } : \Oo( \E(X)) \to \Oo( \E(X)) . 
$$
Note that $\partial_{t^k 1}$ is of cohomological degree $2k-2$.
\begin{lemma}
For any $k \ge 0$, the operator $\partial_{t^k 1}$ commutes with the operators $Q = \dbar - t \partial$, $\partial_{P(\eps,L)}$ and $\Delta_\eps$.   Further all of these operators preserve the subspace $\Ool(\E)$ of local functionals.
\end{lemma}
\begin{proof}
Indeed, $\Delta_\eps$, $\partial_{P(\eps,L)}$ and $\partial_{t^k 1}$ are all constant-coefficient differential operators on $\Oo(\E(X))$, which of course automatically commute.   Any element of $\E(X)$, when viewed as a derivation of $\Oo(\E(X))$, preserves the space of local functionals. Finally, any closed element of $\E(X)$, when viewed as a derivation, commutes with $Q$.   
\end{proof}

In order to discuss the dilaton axiom, we need to define additional derivations of $\Oo(\E(X))$.  Any continuous linear map $ A : \E(X) \to \E(X)$ defines a derivation $\partial_A$ of $\Oo(\E(X))$, characterized by the property that $\partial_A(\alpha) = A^\ast (\alpha)$ for each $\alpha \in \E(X)^\vee$. (Here $A^\ast$ indicates the adjoint of $A$).  

Let
$$
\op{Eu} :  \Oo(\E(X)) \to \Oo(\E(X)) 
$$
be the Euler derivation, defined as the derivation associated to the identity map $\op{Id} : \E(X) \to \E(X)$.  Note that $\op{Eu}( \alpha) = k \alpha$ if $\alpha$ is homogeneous of degree $k$ as a function on $\E(X)$.

\begin{lemma}
The operator  $\op{Eu}$ commutes with $Q$ and satisfies
\begin{align*}
[\op{Eu}, \Delta_\eps ] &= - 2 \Delta_\eps \\
[\op{Eu}, \partial_{P(\eps,L)} ] &= - 2 \partial_{P(\eps,L)}.
\end{align*}
Also, $\op{Eu}$ preserves the subspace $\Ool(\E(X))$ of $\Oo(\E(X))$ consisting of local functionals. 
\end{lemma}
\begin{proof}
This is a trivial calculation. 
\end{proof}
It follows that 
$$
\op{Eu} + 2 \hbar \frac{\partial}{\partial \hbar} : \Oo(\E) [[\hbar]]
$$
commutes with $Q$, $\hbar \partial_{P(\eps,L)}$ and with $\hbar \Delta_\eps$.  

\subsection{}
Now we have constructed derivations 
$$
\partial_{t^k \cdot 1} , \op{Eu} + 2\hbar \frac{\partial}{\partial \hbar }  : \Oo(\E)[[\hbar]] \to \Oo(\E)[[\hbar]]
$$
which commute with $Q$, $\hbar \partial_{P(\eps,L)}$ and $\hbar \Delta_\eps$.    
\begin{lemma}
\label{lemma_derivation_rgflow}
Let 
$$
D : \Oo(\E)[[\hbar]] \to \Oo(\E)[[\hbar]]
$$
be a derivation of cohomological degree $k$ which commutes with $Q$, $\hbar \partial_{P(\eps,L)}$ and $\hbar \Delta_\eps$.    Suppose also that $D$ preserves the subspace of local functionals.  

Let $\{\F[L] \in \Oo(\E)[[\hbar]] \mid L \in \R_{> 0} \}$ be a collection of functionals which satisfy the renormalization group equation,  the quantum master equation, and the locality axiom. Then, so does the collection of functionals
$$
\F[L] + \delta \hbar D \hbar^{-1} \F[L]
$$
where $\delta$ is a parameter of cohomological degree $-k$.
\end{lemma}
\begin{proof}
The renormalization group flow operator
$$
\EA{P(\eps,L), \F } = \hbar \log \left( \exp ( \hbar \partial_{P(\eps,L)}  ) \exp ( \F / \hbar\right) 
$$
is defined entirely in terms of the differential operator $\hbar \partial_{P(\eps,L)}$.  The quantum master equation can be written as 
$$
(Q + \hbar \Delta_L ) \exp \left( \F[L] / \hbar \right) = 0.
$$
Thus, the QME is defined entirely in terms of $Q$ and $\hbar \Delta_L$. 

Now,  suppose that $\{\F[L] \in \Oo(\E(X))[[\hbar]] \mid L \in \R_{> 0}\}$ is a collection of functionals satisfying the renormalization group equation.
\begin{align*}
\exp \left( \hbar^{-1} \EA{P(\eps,L), \F[\eps] + \delta \hbar D \hbar^{-1} \F[\eps]  }\right) &= e^{\hbar \partial_{P(\eps,L)} }  \exp\left( \hbar^{-1} \F[\eps]  + \delta D \hbar^{-1} \F[\eps] \right)\\
&=  e^{\hbar \partial_{P(\eps,L)}  }  ( 1 + \delta D ) \exp \left( \F[\eps] / \hbar ) \right) \\
&=  ( 1 + \delta D ) e^{\hbar \partial_{P(\eps,L)}  }  \exp \left( \F[\eps] / \hbar ) \right) \\
&= ( 1 + \delta D ) \exp \left( \F[L] / \hbar \right).
\end{align*}
Thus, $\F[L] + \delta \hbar D  \hbar^{-1} \F[L]$ satisfies the renormalization group equation.  The proof that it satisfies the quantum master equation is identical. 
\end{proof}

\subsection{}
Thus, the operators $\partial_{t^k \cdot 1}$ and 
$$
\op{Eu} + 2 \hbar^2 \frac{\partial}{\partial \hbar } \hbar^{-1}
$$
all preserve the quantum master equation and the renormalization group equation.   

The dilaton equation will say that a quantization $\F[L]$ is fixed by certain combinations of these operators.  Let
$$
\partial_{Dil} = \partial_{t \cdot 1 } - \op{Eu}.
$$
Note that $\partial_{Dil} - 2 \hbar^2 \frac{\partial}{\partial \hbar } \hbar^{-1} $ preserves the renormalization group equation and the quantum master equation, and so acts on the space of theories.

\begin{lemma}
The operator $\partial_{Dil} - 2 \hbar^2 \frac{\partial}{\partial \hbar } \hbar^{-1} $ fixes the classical BCOV action functional $I \in \Ool(\E(X))$. 
\end{lemma}
\begin{proof}
The classical BCOV action was defined using integrals of $\psi$ classes over $\mbar_{0,n}$.  The fact that these integrals satisfy the dilaton equation immediately implies the lemma. 
\end{proof}
\begin{corollary}
Let $\{\F[L]\}$ be a quantization of the BCOV theory.  Then 
$$\{\F[L] + \delta \left(\partial_{Dil}-  2 \hbar^2 \frac{\partial}{\partial \hbar } \hbar^{-1} \right)  \F[L]  \}$$ 
is a quantization of the BCOV theory over $\C[\delta]/  \delta^2$ where $\delta$ is of cohomological degree $0$. 
\end{corollary}
\begin{proof}
We have already seen that $\F[L] + \delta  \left( \partial_{Dil} - 2 \hbar^2 \frac{\partial}{\partial \hbar } \hbar^{-1} \right)  \F[L] $ satisfies the renormalization group equation, quantum master equation, and locality axiom.  Because the classical interaction $I$ satisfies the dilaton equation we see that, modulo $\hbar$, 
$$\lim_{L \to 0} \F_0[L] + \delta  \left( \partial_{Dil} - 2 \hbar^2 \frac{\partial}{\partial \hbar } \hbar^{-1} \right)  \F_0[L]   = I.$$  The remaining (dimension) axioms we imposed on a quantization are immediate. 
\end{proof}

\begin{definition}
A quantization $\{\F[L]\}$ of BCOV theory satisfies the \emph{dilaton axiom} if the family of quantizations $\F[L] + \delta \left(\partial_{Dil} - 2 \hbar^2 \frac{\partial}{\partial \hbar } \hbar^{-1}\right)  \F[L]$ over $\C[\delta]/\delta^2$ is homotopic to the trivial family $\F[L]$.
\end{definition}

\subsection{}
This definition is a little abstract.  Let us make it more concrete.   There are several (equivalent) ways of saying what it means to have a homotopy of the family $\{\F[L] + \delta \left(\partial_{Dil} -2 \hbar^2 \frac{\partial}{\partial \hbar } \hbar^{-1} \right) \F[L]\}$ of quantizations with the trivial family.  One is simply to have a family of quantizations over $\Omega^\ast ([0,1]) \otimes \C[\delta]/ \delta^2$, which modulo $\delta$ is a trivial family, and which specializes to $\{\F[L] + \delta \left(\partial_{Dil}-2 \hbar^2 \frac{\partial}{\partial \hbar } \hbar^{-1} \right) \F[L]\}$ at $t = 0$, and to the trivial family $\{\F[L]\}$ at $t = 1$.

Because we are working modulo $\delta^2$, we can use an alternative (and equivalent) definition of homotopy based on the familiar notion of cochain homotopy.   With this definition, a homotopy is given by a collection $\mbf{G}[L] \in \hbar \Oo(\E)[[\hbar]]$ of functionals such that $\F[L] +  \eps \delta \mbf{G}[L]$ satisfies the renormalization group equation and locality axiom (where $\eps$ is a parameter of cohomological degree $1$), and 
$$
Q \mbf{G}[L] + \hbar \Delta_L \mbf{G}[L] + \{F[L], \mbf{G}[L] \}_ L = \left( \partial_{Dil}-2 \hbar^2 \frac{\partial}{\partial \hbar } \hbar^{-1} \right) \F[L]. 
$$

\subsection{}

Let us now discuss the string equation, which is similar (but more complicated) to the dilaton equation. 

Let
$$
\partial_{t^{-1} \star }   : \Oo(\E(X)) \to \Oo(\E(X))
$$
be the derivation associated to the linear map sending 
$\phi \in \E(X) = \PV(X)[[t]][2]$ to $t^{-1} \phi$, if $\phi \in t \PV(X)[[t]]$, and to zero if $\phi \in \PV(X)$. 

Recall that the dilaton axiom stated that each $\F_g[L]$ satisfies a differential equation which is independent of $L$.  The reason that the string axiom is a bit more delicate is that the equation satisfied by $\F_g[L]$ depends on $L$. 

We let
$$
\partial_{Str}[0] = \partial_{1} - \partial_{t^{-1} \star } 
$$
where $\partial_1$ is the derivation associated to the element $1 \in \PV^{0,0}(X) \subset \PV(X)[[t]]$.    This is the scale $0$ string operator.

Let us define a derivation $Y[L] : \Oo(\E(X)) \to \Oo(\E(X))$ as the derivation associated to the linear map
$$
t^k \alpha \mapsto \partial_{k = 0} \dbar^\ast \partial \int_{u = 0}^L e^{- u[\dbar, \dbar^\ast] } \alpha. 
$$
Let
$$
\partial_{Str}[L] = \partial_{Str}[0] + Y[L]. 
$$
This is the scale $L$ string operator.

Suppose that $\{\F_g[L]\}$ is a quantization of the BCOV theory on $X$.     Then, define a family of functionals 
$$(\partial_{Str} \F_g ) [L] \in \Oo(\E(X)) \otimes \C[\delta] / \delta^2$$ by
$$
(\partial_{Str} \F_g ) [L] = 
\begin{cases}
(\partial_{Str} \F_0) [L] &= \F_0[L] \\
(\partial_{Str} \F_g)[L] &= \F_g[L] + \delta \partial_{Str}[L] \F_g[L] \text{ if } g > 0.  
\end{cases}
$$
(Here, $\delta$ is a parameter of cohomological degree $2$ and square $0$).

\begin{proposition}
Suppose that $\{\F[L]\}$ is a quantization of the BCOV theory on $X$.  Then, $\{ (\partial_{Str} \F )[L]\}$ is a family of quantizations over $\C[\delta]/\delta^2$. 
\end{proposition}
\begin{proof}
This is a reasonably straightfoward computation.  (More details to come in a later version).
\end{proof}

\subsection{}
Now suppose that $\{\F[L]\}$ is a quantization of BCOV theory on $X$ which satisfies the string and dilaton equations.  Let 
$$\mc{H} \left( \E(X) \right) = \op{Ker} [\dbar, \dbar^\ast ] \subset \E(X)$$
be the subspace of harmonic fields.   Note that the Hodge lemma implies that the operator $\partial$ is zero on the subspace $\mc{H} \left( \E (X)\right)$.  Thus, there is an isomorphism of cochain complexes
$$
\mc{H}\left( \E (X) \right) =  \oplus  H^\ast ( X, \wedge^\ast TX )  [[t]] [2].
$$
Note that, in addition, the BV operator $\Delta_\infty$ is zero.       Let
$$
\F^{\mc{H}} = \F[\infty] \mid_{\mc{H}(\E(X))} \in \Oo (\mc{H}(\E(X)) )[[\hbar]]
$$
be the restriction of $\F[\infty]$ to the space of harmonic fields.  This object allows us to define the \emph{correlators} for the quantization $\{\F[L]\}$ by
$$
\sum \hbar^g \ip{t^{k_1} \alpha_1, \ldots, t^{k_n} \alpha_n }^{\F}_{g,n}  =\left( \frac{\partial}{\partial (t^{k_1} \alpha_1 ) } \cdots \frac{\partial}{\partial (t^{k_1} \alpha_1 ) }  \F[\infty] \right) (0) \in \C[[\hbar]],
$$
for $\alpha_i \in H^\ast (X, \wedge^\ast TX)$.

The following lemma will be proved in section \ref{section givental}.
\begin{lemma}
The correlators $\ip{-}^{\F}_{g,n}$ do not change if we change the quantization $(\{\F[L]\}, g)$ to a homotopic quantization $(\{\F'[L]\}, g')$.  
\end{lemma}
As we will see shortly, the correlators do depend on the choice of a complementary filtration to the Hodge filtration on $X$.  The correlators as defined above correspond to the complex conjugate to the Hodge filtration.  

\begin{lemma}
If the quantization $\{\F[L]\}$ satisfy the string and dilaton equation in the homotopical sense described above, then the correlators satisfy the strict string and dilaton equation:
\begin{align*}
\ip{t \cdot 1, t^{k_1} \alpha_1, \dots, t^{k_n} \alpha_n   }^{\F}_{g,n+1,\beta} &= (2g - 2 + n )\ip{t^{k_1} \alpha_1, \dots, t^{k_n} \alpha_n   }_{g,n,\beta}   \\
\ip{1, t^{k_1} \alpha_1, \dots, t^{k_n} \alpha_n   }^{\F}_{g,n+1,\beta} &= \sum_{i = 1}^n \ip{t^{k_1} \alpha_1, \dots, t^{k_i - 1} \alpha_i , \dots, t^{k_n} \alpha_n   }_{g,n,\beta}  
\end{align*}
\end{lemma}

\section{BCOV theory and the Givental formalism}
\label{section givental}

In \cite{Giv01,CoaGiv01,Giv04} Givental and Coates showed that the generating function for the Gromov-Witten of a variety $X$ is most naturally viewed as a state in a Fock space constructed from the cohomology of $X$.   In genus $0$, Givental and Coates' work on the $A$-model is paralleled in the $B$-model by the work of Barannikov \cite{Bar00,Bar99}. 

 In \cite{Cos05} the first-named author showed that the partition function of any topological (conformal) field theory is naturally a state in a certain Fock space associated to an infinite-dimensional symplectic vector space. For the categorical $B$-model, where the field theory is constructed from the dg category $\op{Per}(X)$ of perfect complexes on a variety $X$, the symplectic vector space is the periodic cyclic homology of the category $\op{Per}(X)$.   

In this section, we will show that the partition function for the BCOV theory can be viewed as a state in a Fock space constructed from an infinite-dimensional symplectic vector space.  The symplectic vector space that appears is isomorphic to that coming from the categorical $B$-model, which allows one to compare the partition function of the categorical $B$-model with that of the BCOV theory.

We conjecture that the partition function for the BCOV theory on the elliptic curve coincides with that of the categorical $B$-model on the elliptic curve.    

\subsection{}
Let us start by explaining the symplectic vector space which relates to the BCOV theory.   Let $X$ be a compact Calabi-Yau manifold of dimension $d$.   As a cochain complex, this symplectic vector space is 
$$\mc{S}(X) = \PV(X)((t))[2],$$ with differential $\dbar - t \partial$.    The symplectic pairing on $\mc{S}(X)$ is
$$
\omega( \alpha f(t), \beta g(t)  ) =\left( \op{Res}_{t} f(t) g(-t) \d t \right) \op{Tr} (\alpha \wedge \beta). 
$$
Note that the symplectic pairing is of cohomological degree $6 -2d$.    Further, the subspace 
$$\mc{S}_+(X) = \PV(X)[[t]][2] = \E(X) $$
is Lagrangian.  Thus, the space of fields of the BCOV theory is a Lagrangian subspace of this symplectic vector space.

\subsection{}
We want to construct the Fock space for the differential graded symplectic vector space $\mc{S}(X)$.   Let us start by recalling the construction of the Fock module for a finite dimensional symplectic vector space. 
\begin{definition}
Let $V$ be a dg symplectic vector space with symplectic form $\omega$.  Let 
$$
\omega^{-1} \in \wedge^2 V
$$
denote the inverse to the symplectic form. Then, the Weyl algebra $\mc{W}(V)$ is the pro-free dg algebra generated by $V^\vee$, and a parameter $\hbar$, subject to the relation that
$$
[a,b] = \hbar \omega^{-1}(a,b)
$$
for all $a,b \in V^\vee$. 
\end{definition} 
\begin{definition}
 Let $L \subset V$ be a Lagrangian sub-cochain complex of a dg symplectic vector space $V$.  The Fock module $\mc{F}(L)$ is defined to be the quotient of $\mc{W}(V)$ by the left ideal generated by 
$$\op{Ann}(L) \subset V^\vee \subset \mc{W}(V),$$
where $\op{Ann}(L)$ is the annihilator of $L$. 
\end{definition}
Note that modulo $\hbar$, $\mc{W}(V)$ is the completed symmetric algebra $\what{\Sym}^\ast V^\vee$.  Further, $\mc{W}(V)$ is a flat $\C[[\hbar]]$ module, so that there is a non-canonical isomorphism 
$$\mc{W}(V) \iso \left(\what{\Sym}^\ast V^\vee \right)[[\hbar]].$$ 

Let us choose a complementary Lagrangian $L' \subset V$ to $L$.  We do not, however, assume that $L'$ is preserved by the differential. Thus, $V = L \oplus L'$, and the symplectic pairing induces an isomorphism $L' = L^\vee$.     

Once we choose such an isomorphism, we get a splitting of the map $V^\vee \to L^\vee$, and so a map of algebras
$$
\Oo(L) [[\hbar ]] = \what{\Sym}^\ast L^\vee [[\hbar]] \to \mc{W}(V). 
$$
By composing this map with the quotient map $\mc{W}(V) \to \mc{F}(L)$,  we find an isomorphism of graded vector spaces
$$
\Oo(L) [[\hbar]] \iso \mc{F}(L).
$$
This is not, however, an isomorphism of cochain complexes; for later purposes, it will be helpful to describe the differential on $\Oo(L)[[\hbar]]$ corresponding to that on $\mc{F}(L)$.

Define a pairing $P$ on $L'$ by
$$
P(l'_1, l'_2) = \omega (l_1, \d l_2 ) .
$$
Because the symplectic pairing identifies $L'$ with the dual of $L$, we can view $P$ as an element of $\Sym^2 L$.   Let $\partial_P : \Oo(L) \to \Oo(L)$ be the differential operator corresponding to $P$.  
\begin{lemma}
Under the isomorphism of graded vector spaces $\Oo(L) [[\hbar]] \iso \mc{F}(L)$, the differential $\d_{\mc{F}(L)}$ on $\Oo(L)[[\hbar]]$ is 
$$
\d_{\mc{F}(L)}  = \d_L + \hbar \partial_P.
$$
Here, $\d_L$ is the derivation of $\Oo(L)$ induced from the differential on $L$. 
\label{lemma fock space differential}
\end{lemma}
\begin{proof}
This is a simple calculation.  Indeed,  let $\vac \in \mc{F}(L)$ be the vacuum vector, that is, the image of $1 \in \mc{W}(V)$ under the map $\mc{W}(V) \to \mc{F}(L)$.  Let $\alpha_1,\ldots, \alpha_n \in L^\ast = L'$.  Then, the space $\mc{F}(L)$ is spanned by $\alpha_1 \dots \alpha_n \vac$, where we view the $\alpha_i$ as elements of $\mc{W}(V)$. 

Now, 
$$
\d ( \alpha_1 \dots \alpha_n \vac ) = \sum \pm \alpha_1 \dots (\d \alpha_i ) \dots \alpha_n \vac.
$$
Recall that $L^\ast \subset V^\ast$ is not preserved by the differential (indeed, it is identified with $L' \subset V$ under the isomorphism $V \iso V^\ast$).  Let us decompose the differential $\d : L^\ast \to L^\ast \oplus (L')^\ast$ into $\d = \d_1 \oplus \d_2$, where $\d_1$ lands in $L^\ast$ and $\d_2$ lands in $(L')^\ast$.  Then, we find that
\begin{multline*}
\d (\alpha_1 \dots \alpha_n \vac ) = \sum \pm \alpha_1 \dots (\d_1 \alpha_i) \dots \alpha_n \vac \\ + \sum_{1 \le i < j\le n} \pm [\d \alpha_i, \alpha_j ] \alpha_1 \dots \what{\alpha_i} \dots \what{\alpha_j} \dots \alpha_n \vac
\end{multline*}
were, as usual, $\what{\alpha_i}$ indicates that we omit $\alpha_i$, and $\pm$ indicates the usual Koszul sign.  

Now, 
$$
[\d \alpha_i, \alpha_j ] = \hbar \omega^{-1} ( \d \alpha_i, \alpha_j ) .
$$
Thus, in the expression for $\d ( \alpha_1 \dots \alpha_n \vac)$ presented above, the first line corresponds to the differential $\d_L$ on $\Oo(L)[[\hbar]]$, and the second term is $\hbar \partial_P$.  
\end{proof}

\subsection{}
We want to mimic these definitions using the infinite-dimensional symplectic vector space $\mc{S}(X)$ and the Lagrangian subspace $\mc{S}_+(X) \subset \mc{S}(X)$.  We immediately run into a problem, however. Although the symplectic form $\omega$ on $\mc{S}(X)$ has no kernel, the inverse $\omega^{-1}$ does not exist.  

Indeed, the symplectic form on $\mc{S}(X)$ is built using the trace pairing 
$$
\ip{\alpha,\beta} = \op{Tr}(\alpha \beta) 
$$
on the space $\PV(X)$ of polyvector fields on $X$.  The inverse to this trace pairing is a distributional polyvector field 
$$
\ip{-,-}^{-1} \in \br{\PV}(X) \otimes \br{\PV}(X).
$$
 It follows that the inverse to the symplectic form on $\mc{S}(X) = \PV(X)((t))[2]$ will also be of a distributional nature.

This is not a major problem, however, because we can instead use any element of $\PV(X) \otimes \PV(X)$ which is cohomologous to $\ip{-,-}^{-1}$.  

Let $g$ be a K\"ahler metric on $X$, and let $L \in \R_{> 0}$.  The heat kernel 
$$
K_L^g \in \PV(X) \otimes \PV(X)
$$
is an inverse up to homotopy for the trace pairing.  Indeed, the scale $0$ heat kernel $K_0 \in \br{\PV}(X)^{\otimes 2}$ is the actual inverse for the trace pairing, and the expression
$$
\int_0^L (\dbar^\ast \otimes 1) K^g_t \d t \in \br{\PV}(X)^{\otimes 2}
$$
provides a cochain homotopy between $K_0$ and $K_L^g$.

The inverse to the symplectic form $\omega$ on $\mc{S}(X)$ is 
$$
\omega^{-1} = \sum_{k \in \Z} K_0 (-1)^k t^k \otimes t^{-1-k} \in \br{\mc{S}}(X) ^{\otimes 2}
$$
(where $\br{\mc{S}}(X)$ refers to the distributional completion $\br{\PV}(X)((t))[2]$ of $\mc{S}(X)$, and the tensor product is, as always, the completed projective tensor product). Thus, a homotopy inverse to $\omega$ is given by the formula
$$
\omega_{g,L}^{-1} = \sum_{k \in \Z} K_L^g (-1)^k t^k \otimes t^{-1-k} \in \mc{S}(X) ^{\otimes 2} 
$$

\subsection{}
We can define the Weyl algebra and the Fock module for $\mc{S}(X)$ using the homotopy inverse $\omega_{g,L}$ as follows.  
\begin{definition}
The Weyl algebra $\mc{W}( \mc{S}(X), g,L )$ associated to $\mc{S}(X)$, the metric $g$, and $L \in \R_{> 0}$, is the quotient of the tensor algebra
$$
\left( \prod_{n \in \Z_{\ge 0}} \left( \mc{S}(X)^\vee \right)^{\otimes n} \right)  \otimes \C[[\hbar]]
$$
by the topological closure of the two-sided ideal generated by
$$
[a,b] = \omega_{g,L}^{-1}(a,b) \hbar
$$
for $a,b \in \mc{S}(X)^\vee$. 
\end{definition}
Here, since $\omega_P^{-1}$ is of cohomological degree $2d - 6$, the parameter $\hbar$ is of cohomological degree $6-2d$.  

Note that we can use the symplectic pairing on $\mc{S}(X)$ to identify $\mc{S}(X)^\vee$ with $\br{PV}(X) ((t))[2d-4]$.  

\begin{definition}
The Fock space $\mc{F}  (\mc{S}_+(X), g, L)$ is the quotient of $\mc{W}( \mc{S}(X), g, L)$ by the left ideal generated topologically by the subspace
$$
\op{Ann}(\mc{S}_+(X), g, L) \subset \mc{S}(X)^\vee.
$$ 
\end{definition}

The fact that $\mc{S}_+(X)$ is a subcomplex implies that this left ideal is preserved by the differential $\dbar - t \partial$ on $\mc{W}(\mc{S}(X), g, L)$. 

\begin{lemma}
The Weyl algebra $\mc{W}(\mc{S}(X), g,L)$ and the Fock module $\mc{F}(\mc{S}_+(X), g, L)$ are independent, up to homotopy, of the choice of $g$ and $L$.   is independent, up to homotopy, of the choice of $P$.
\end{lemma}

\begin{proof}
Indeed, suppose we have a family of K\"ahler metrics $g_t$  parametrized by the $k$-simplex $\tr^k$.  We will show how to construct a family of dg algebras $\mc{W} ( \mc{S}(X), g_t, \d t, L , \d L)$ over $\Omega^\ast( \tr^k \times (0,\infty)_L)$ which specializes to $\mc{W}( \mc{S}(X), g_t, L)$ at $(t,L) \in \tr^k \times (0,\infty)$.  

To do this, we will construct a family of homotopy inverses 
$$
\omega^{-1}_{g_t, \d t, L, \d L} \in \mc{S}(X)^{\otimes 2} \otimes \Omega^\ast( \tr^k \times (0,\infty) )
$$
to the symplectic form $\omega$ on $\mc{S}(X)$.   This will be constructed from a family of heat kernels.  

Let
$$
K_{g_t, \d t, L} \in \PV(X)^{\otimes 2} \otimes \Omega^\ast( \tr^k \times (0,\infty)_L ) 
$$
be the heat kernel for the $\Omega^\ast( \tr^k \times (0,\infty)_L )$-linear differential operator
$$
[\dbar + \d_{dR}, \dbar^\ast_{g_t} ]
$$
where $\d_{dR}$ is the de Rham differential on $\Omega^\ast( \tr^k \times (0,\infty)_L )$.  Here 
$$
\dbar^\ast_{g_t} : \PV(X) \to \PV(X) \otimes \cinfty( \tr^k ) 
$$
is the family of operators associated to the metric $g_t$, for $t \in \tr^k$. 

Let 
$$
K_{g_t, \d t, L, \d L} = K_{g_t, \d t, L} + \d L \dbar^\ast K_{g_t, \d t, L} \in \PV(X)^{\otimes 2} \otimes \Omega^\ast( \tr^k, (0,\infty)_L ).
$$
This is a closed element, and provides a cochain homotopy between the various heat kernels. 

Let
$$
\omega^{-1}_{g_t, \d t, L, \d L} = K_{g_t, \d t, L, \d L}  \sum_{k \in \Z} (-1)^k t^k \otimes t^{-1-k} \in \mc{S}(X)^{\otimes 2} \otimes \Omega^\ast( \tr^k \times (0,\infty)_L ) .
$$

We define the Weyl algebra
$$
\mc{W}(\mc{S}(X),  g_t, \d t, L, \d L  ) 
$$
to be the $\Omega^\ast( \tr^k \times (0,\infty)_L )$-algebra which is a quotient of the tensor algebra 
$$\Omega^\ast( \tr^k \times (0,\infty)_L )  \otimes \C[[\hbar]] \otimes \prod_{n \ge 0} \left( \mc{S}(X)^\vee \right)^{\otimes n}$$
by the two-sided ideal generated by the topological closure of the relations
$$
[a,b] = \hbar \omega_{g,L, \d L}^{-1} ( a,b) .
$$
As before, we can define the Fock module $\mc{F} ( \mc{S}(X), g_t, \d t, L , \d L ) $  as the quotient of $\mc{W}(\mc{S}(X), g, L, \d L)$  by the left ideal generated by $\op{Ann}( \mc{S}_+(X))$. 

Thus, $\mc{W}(\mc{S}(X),  g_t, \d t, L, \d L  )$ defines a differential graded algebra over $\Omega^\ast( \tr^k \times (0,\infty)_L )$, specializing to the Weyl algebra $\mc{W}(\mc{S}(X),  g_t,  L )$ at $(t, L) \in \tr^k \times (0,\infty)_L$.  Similarly, $\mc{F} ( \mc{S}(X), g_t, \d t, L , \d L )$ defines a dg module over $\mc{W}(\mc{S}(X),  g_t, \d t, L, \d L  )$, flat over $\Omega^\ast( \tr^k \times (0,\infty)_L ) $, which specializes to the Fock module $\mc{F} ( \mc{S}(X), g_t, L)$ at $(t,L)$. 
\end{proof}

It follows that the cohomology of $\mc{W}(\mc{S}(X), g, L)$ is independent of $g$ and $L$.  Thus, we will refer to the cohomology of $\mc{W}(\mc{S}(X), g, L)$ are simply $H^\ast \mc{W}( \mc{S}(X))$.  Similarly, we will use the notation $H^\ast \mc{F}( \mc{S}_+(X)) $ to refer to the cohomology of $\mc{F}( \mc{S}(X), g, L)$.  

Note that (by the degeneration of the Hodge to de Rham spectral sequence) the map
$$
H^\ast \mc{S}_+(X) \to H^\ast \mc{S}(X)
$$
is injective, and the image is Lagrangian.  Thus, we can construct the Weyl algebra $\mc{W}( H^\ast(\mc{S}(X))$, and the Fock module $\mc{F} ( H^\ast\mc{S}_+(X))$.  
\begin{lemma}
There are canonical isomorphisms
\begin{align*}
\mc{W}( H^\ast(\mc{S}(X)) &\iso H^\ast \mc{W} ( \mc{S}( X) ) \\
\mc{F} ( H^\ast\mc{S}_+(X)) &\iso H^\ast  \mc{F}( \mc{S}_+(X)) .
\end{align*}
These isomorphisms are compatible with the algebra and module structures present. 
\end{lemma}
\begin{proof}
Indeed, there's a cochain map
$$
\mc{S}(X) \to \mc{W}(\mc{S}(X)).
$$
Passing to cohomology yields a map $H^\ast \mc{S}(X) \to H^\ast \mc{W}(\mc{S}(X)).$  The relations in the Weyl algebra $\mc{W}( H^\ast \mc{S}(X) )$ imply that this map extends uniquely to an algebra homomorphism 
$$
\mc{W}( H^\ast(\mc{S}(X)) \to H^\ast \mc{W} ( \mc{S}( X) ).
$$
This homomorphism is an isomorphism. A similar argument proves the result about the Fock modules. 
\end{proof}
\subsection{}
The relationship between the BCOV theory and the symplectic formalism developed in this section is provided by the following proposition.
\begin{proposition}
\label{proposition fock qme}
Let $g$ be a K\"ahler metric on $X$, and let $L \in \R_{> 0}$.  Then, there is an isomorphism of cochain complexes
$$
\mc{F}( \mc{S}_+(X), g, L ) \iso \left( \Oo(\mc{S}_+(X)[[\hbar]] , \dbar - t \partial + \hbar \Delta_L \right).
$$
\end{proposition}
\begin{proof}
The proof is almost identical to that of lemma \ref{lemma fock space differential}, which is the finite dimensional analog of this proposition.  

In the course of this proof, we will omit the parameters $g$ and $L$ from the notation for the Weyl algebra and the Fock space.    Thus, $\mc{W}(\mc{S}(X))$ will indicate the Weyl algebra associated to the metric $g$ and $L$, and similarly $\mc{F}( \mc{S}_+(X) )$ will refer to the Fock module for $g$ and $L$. 

The first step in the proof is to construct an isomorphism of graded vector spaces
$$
\Oo(\mc{S}_+(X) )[[\hbar]] \iso \mc{F}(\mc{S}_+(X)).
$$
Then we will verify that this isomorphism is compatible with the differentials.

Let
$$
\mc{S}_-(X) = t^{-1} \PV(X) [t^{-1} ][2] \subset \mc{S}(X).
$$
Note that we have a direct sum decomposition
$$
\mc{S}(X) = \mc{S}_+(X) \oplus \mc{S}_-(X),
$$
and that both subspaces are Lagrangian.  However, $\mc{S}_-(X)$ is \emph{not} preserved by the differential $\dbar - t \partial$ on $\mc{S}(X)$.  

The choice of complementary Lagrangian $\mc{S}_-(X)$ gives a splitting of the inclusion $\mc{S}_+(X) \into \mc{S}(X)$, and so (dually) a splitting 
$$
\Phi : \mc{S}_+(X)^\vee \to \mc{S}(X)^\vee
$$ 
of the projection $\mc{S}(X)^\vee \to \mc{S}_+(X)^\vee$.

However, because $\mc{S}_-(X)$ is not a subcomplex of $\mc{S}(X)$, the map $\Phi$ is not a cochain map, but simply a map of graded vector spaces.  

Now, elements of $\mc{S}_+(X)^\vee$ commute in the Weyl algebra generated by $\mc{S}(X)^\vee$.  It follows that the map $\Phi$ extends to an algebra homomorphism
$$
\Phi : \Oo( \mc{S}_+(X) ) [[\hbar]] \to \mc{W}(\mc{S}(X)).
$$
However, this map is not compatible with differentials.

When we compose this algebra homomorphism $\Phi$ with the natural quotient map
$$
\mc{W}( \mc{S}(X) ) \to \mc{F}( \mc{S}_+(X) )
$$
we find the desired isomorphism of graded vector spaces
$$
\Oo(\mc{S}_+(X) ) [[\hbar]] \iso \mc{F}( \mc{S}_+(X) ).
$$

It remains to verify that, under this isomorphism, the differential on $\mc{F}( \mc{S}_+(X), g, L)$ corresponds to the differential
$$
\dbar - t \partial + \hbar \Delta_L
$$
on $\Oo(\mc{S}_+(X) )[[\hbar]]$.  

Let $\vac \in \mc{F}(\mc{S}_+(X))$ denote the vacuum vector, that is, the image of $1 \in \mc{W}(\mc{S}(X))$ under the quotient map $\mc{W}(\mc{S}(X)) \to \mc{F}(\mc{S}_+(X))$.    Let $\d_{\mc{F}}$ denote the differential on $\mc{F}(\mc{S}_+(X))$, and $\d_{\mc{W}}$ that on $\mc{W}(\mc{S}(X))$.   Let $\alpha_1, \ldots, \alpha_n \in \mc{S}_+(X)^\vee$.  What we need to verify is that
$$
\d_{\mc{F}} \left( \Phi(\alpha_1) \cdots \Phi(\alpha_n)  \right) \vac = \Phi \left(  (\dbar - t \partial + \hbar \Delta_L ) \alpha_1 \cdots \alpha_n \right) \vac.
$$
We can use the symplectic pairing on $\mc{S}(X)$ to identify
\begin{align*}
\mc{S}(X)^\vee &= \br{\mc{S}}(X) = \br{\PV}(X) ((t)) \\
\mc{S}_+(X)^\vee &= \br{\mc{S}}_-(X) = t^{-1} \br{\PV}(X) [t^{-1}]\\
\mc{S}_-(X)^\vee &= \br{\mc{S}}_+(X) =  \br{\PV}(X) [[t]].
\end{align*}
Under these identifications, the map $\Phi$ is the natural map 
$$
\Phi : t^{-1} \br{\PV}(X) [t^{-1}] \to \br{\PV}(X)((t)).
$$
Now, it is clear that
\begin{align*}
\Phi ( \dbar \alpha ) &= \dbar \Phi(\alpha) \\
\Phi ( \partial \alpha) &= \partial \Phi(\alpha) ,
\end{align*}
for all $\alpha \in \mc{S}_+(X)^\vee$. However, 
$$
t \Phi(\alpha) \neq \Phi (t \alpha),
$$
because the subspace $\mc{S}_+(X)^\vee \subset \mc{S}(X)^\vee$ is not preserved by the operator $t$.  

To prove the result, we thus need to verify that
$$
\sum \pm \Phi( \alpha_1) \cdots \left(  [t , \Phi ] \partial \alpha_i \right) \cdots \Phi (\alpha_n) \vac = \Phi ( \hbar \Delta_L \alpha_1 \cdots \alpha_n ) \vac.
$$
The operator $[t,\Phi]$ maps $\mc{S}_+(X)^\vee$ to $\mc{S}_-(X)^\vee$. Further, the action of any element of $\mc{S}_-(X)^\vee$ on the vacuum vector $\vac \in \mc{F}(\mc{S}_+(X))$ is zero.  Thus, 
$$
\sum \pm \Phi( \alpha_1) \cdots \left(  [t , \Phi ] \partial \alpha_i \right) \cdots \Phi (\alpha_n) \vac = \sum \pm [\Phi(\alpha_i) , [t,\Phi] (\partial \alpha_j ) ] \prod_{k \neq i,j} \Phi (\alpha_k) \vac.
$$
From this formula, we see that to prove the result we need to verify that, for all $\alpha, \beta \in \mc{S}_+(X)^\vee$, 
$$
[\Phi(\alpha) , [t,\Phi] (\partial \beta ) ] = \hbar \Delta_L ( \alpha \beta) \in \C[[\hbar]].
$$
In other words, we need to verify that 
$$
\omega_{g,L}^{-1} \left( \Phi(\alpha), t \partial \Phi(\beta )  \right)  = \Delta_L (\alpha \beta) \in \C.
$$
This is immediate, since $\Delta_L$ is defined in terms of $(\partial \otimes 1) K_L$, and $\omega_{g,L}^{-1}$ is defined in terms of $K_L$.  
\end{proof}

\subsection{}
Now suppose that we have a quantization $\F = (\{\F[L]\},g))$ of the BCOV theory on a Calabi-Yau $X$ (where $g$ is a K\"ahler metric on $X$).  Let
$$
Z_{\F} [L] = \exp ( \F[L] / \hbar ) \in \Oo ( \E(X)) ((\hbar)) = \Oo (\mc{S}_+(X)) ((\hbar)).
$$ 
 Then, the quantum master equation for $\F[L]$ says that 
$$
(\dbar - t \partial + \hbar \Delta_L ) Z_{\F}[L] = 0.
$$
Since this differential is the same as the differential on the Fock module $\mc{F}(\mc{S}_+(X))$, we see that
\begin{corollary}
$Z_{\F}[L]$ defines a state in the Fock space $\mc{F}( \mc{S}_+(X), g ,L) [\hbar^{-1}]$.   
\end{corollary}
Thus, the partition function for the BCOV theory is a state in the Fock space modelled on the infinite dimensional vector space $\mc{S}(X)$. 

Recall that the cohomology of $\mc{F}( \mc{S}_+(X), g ,L)$ is independent of $g$ and $L$, and can be identified with the Fock space $\mc{F} (H^\ast \mc{S}_+(X) )$ for the cohomological symplectic vector space $H^\ast \mc{S}(X)$. 
\begin{lemma}
The cohomology class of $Z_{\F}[L]$ in $\mc{F} (H^\ast \mc{S}_+(X) ) [\hbar^{-1}]$ is independent of $L$.  We will refer to this cohomology class as 
$$
[Z_{\F}] \in \mc{F} (H^\ast \mc{S}_+(X) ) [\hbar^{-1}].
$$
\end{lemma}
\begin{proof}

We construct a family $\mc{F} ( \mc{S}_+(X), g, L, \d L)$ of Fock modules over the base ring $\Omega^\ast( (0,\infty)_L)$.   As a graded vector space,
$$
\mc{F} ( \mc{S}_+(X), g, L, \d L) = \Oo(\mc{S}_+(X)) \otimes \Omega^\ast( (0,\infty)_L ) \otimes \C[[\hbar]].
$$ 
The differential on this graded vector space is 
$$
\dbar - t \partial + \hbar \Delta_L + \hbar \partial_{(\partial \dbar^\ast \otimes 1) K_L } .
$$
It is straightforward to check that the renormalization group equation for $\F[L]$ is equivalent to the statement that
$$
Z_{\F}[L] \in\Oo(\mc{S}_+(X)) \otimes \Omega^\ast( (0,\infty)_L ) \otimes \C((\hbar))
$$ 
is closed for this differential.  
\end{proof}

\subsection{}
Now suppose that we have a family of metrics $g_t$ depending smoothly on $t$ in the $k$-simplex.  Suppose that we have a homotopically trivial family of quantizations for this family of metrics.  Such a family of quantizations is given by a collection of functionals
$$
\{ \F[L]  \in \Oo ( \PV(X) ) \otimes \Omega^\ast(\Delta^k ) [[\hbar]] \mid L \in (0,\infty) \},
$$
which satisfy a renormalization group equation, quantum master equation, and locality axiom.  As we have seen in section \ref{section quantization}, by considering such families of quantizations, one can translate a quantization for one metric into a quantization for another. 

As before, we can define
$$
Z_{\F} [L] = \exp ( \F[L] / \hbar ]  \in \Oo ( \PV(X) ) \otimes \Omega^\ast(\Delta^k ) ((\hbar)).
$$
Let $\mc{F}( \mc{S}_+(X), g_t, \d t, L )$ denote the Fock module for the family $g_t$ of K\"ahler metrics.  Recall that this is a dg module over $\Omega^\ast (\Delta^k)$ which specializes to the Fock module $\mc{F}( \mc{S}_+(X), g_t, L )$ at specific values of $t$.  There is an isomorphism of graded vector spaces
$$
\mc{F}( \mc{S}_+(X), g_t, \d t, L ) =  \Oo ( \PV(X) ) \otimes \Omega^\ast(\Delta^k )[[\hbar]].
$$
\begin{lemma}
The partition function $Z_{\F}[L]$, when viewed as an element of 
$$
\mc{F}( \mc{S}_+(X), g_t, \d t, L ) [\hbar^{-1}],
$$
is closed for the natural differential on this space.
\end{lemma}
\begin{proof}
This is the family version of Proposition \ref{proposition fock qme}. The proof is similar. 
\end{proof}

Recall that, in section \ref{section quantization}, we described a simplicial set $\op{Quant}(X)$ of quantizations of the BCOV theory on $X$. Zero simplices in this simplicial set are pairs $(\{\F[L]\}, g)$ where $g$ is a K\"ahler metric on $X$, and $\{\F[L]\}$ is a quantization for this metric.  Higher simplices are given by families of metrics, and homotopies of fields theories, as above.  We will let $\op{Quant}(X)[k]$ denote the set of $k$-simplices of $\op{Quant}(X)$.
\begin{corollary}
The map 
\begin{align*}
\op{Quant}(X)[0] &\to \mc{F} (H^\ast \mc{S}_+(X) ) [\hbar^{-1} ] \\
\left( \{\F[L]\}, g \right) & \mapsto [Z_{\F}]
\end{align*}
descends to a map
$$
\pi_0 \op{Quant}(X) \to \mc{F} (H^\ast \mc{S}_+(X) ) [\hbar^{-1} ]. 
$$
\end{corollary}

\section{Correlation functions and complements to the Hodge filtration}
\label{section correlation}
In this section we will describe how to understand the correlation functions of the BCOV theory from the point of view of the Fock space.  Although some aspects of the Fock space formalism work for non-compact $X$, throughout this section we need to assume that $X$ is compact.  

The correlation functions depend on an additional choice, which we now define.
\begin{definition}
A \emph{polarization} of $H^\ast \mc{S}(X)$ is a Lagrangian subspace
$$
\Lag \subset H^\ast (\mc{S}(X))
$$
such that 
\begin{enumerate}
\item 
$$\Lag \oplus H^\ast \mc{S}_+(X) = H^\ast \mc{S}(X)$$.
\item The subspace $\Lag$ is preserved by the operator 
$$t^{-1} : H^\ast \mc{S}(X)  \to H^\ast \mc{S}(X).$$
\end{enumerate}
\end{definition}
Note that $H^\ast \mc{S}_+(X)$ is preserved by the operator $t$.  Further, the subspace
$$
t^k H^\ast \mc{S}_+(X)  \subset H^\ast \mc{S}(X)
$$
(for $k\in \Z$) define a filtration, whose associated graded can be identified as 
$$
\op{Gr} H^\ast \mc{S}(X) = H^\ast \left( \PV(X) , \dbar \right) ((t)) [2]. 
$$
The choice of a polarization $\Lag$ gives a splitting of this filtration, and so an isomorphism
$$
H^\ast \mc{S}(X) \iso H^\ast (\PV(X), \dbar) ((t))[2].
$$
Explicitly, there's a natural isomorphism
$$
t^k H^\ast \mc{S}_+(X) /  t^{k+1} H^\ast \mc{S}_+(X) \iso t^k H^\ast \mc{S}_+(X) \cap t^{k+1} \Lag. 
$$
Thus, we get an isomorphism
$$
t^k H^\ast \mc{S}_+(X) \cap t^{k+1} \Lag = t^k H^\ast (\PV(X)).
$$

This isomorphism
$$
H^\ast \mc{S}(X) \iso H^\ast (\PV(X)) ((t)) [2]
$$
is symplectic, where $H^\ast (\PV(X)) ((t)) [2]$ is equipped with the symplectic form
$$
\omega ( \alpha t^k, \beta t^l )  = \op{Tr} (\alpha \beta) \op{Res} (t^k (-t)^l \d t ).
$$
Further, the subspace $H^\ast \mc{S}_+(X)$ of $H^\ast \mc{S}(X)$ corresponds to the subspace
$$
H^\ast (\PV(X)) [[t]] [2] \subset H^\ast (\PV(X)) ((t)) [2].
$$
The following lemma is an immediate corollary of these considerations.
\begin{lemma}
Let $\Lag$ be a polarization of $H^\ast \mc{S}(X)$.  Then there is a natural isomorphism
$$
\Phi_{\Lag} : \mc{F} ( H^\ast \mc{S}_+(X) ) \iso \Oo ( H^\ast (\PV(X) )[[t]] [2] ) [[\hbar]].
$$
Here, $\Oo ( H^\ast (\PV(X) )[[t]] [2] )$ denotes the algebra of formal power series on the graded vector space $H^\ast (\PV(X) )[[t]] [2]$.
\end{lemma}
\begin{proof}
Indeed, $\Oo ( H^\ast (\PV(X) )[[t]] [2] )$ is the Fock space for $H^\ast( \PV(X) ) ((t)) [2]$ corresponding to the Lagrangian subsapce $H^\ast (\PV(X)) [[t]] [2]$. 
\end{proof}
Thus, once we choose a polarization $\Lag$, we can consider the partition function $Z_\F$ as an element
$$
\Phi_{\Lag} (Z_\F) \in \Oo (H^\ast (\PV(X) ) [[t]] [2] ) ((\hbar)).
$$
It is not difficult to verify that 
$$
\hbar \log \Phi_{\Lag} (Z_{\F}) \in \Oo (H^\ast (\PV(X) ) [[t]] [2] ) [[\hbar]].
$$
\begin{definition}
Let $\F \in \pi_0( \op{Quant}(X) )$ be the (homotopy class) of a quantization of the BCOV theory on $X$.  Let $\Lag$ be a polarization of $H^\ast \mc{S}(X)$.  Let $\alpha_1, \ldots, \alpha_n \in H^\ast (\PV(X))$.  Define the \emph{correlation functions} associated to $\F$ and $\Lag$ by
$$
\sum \hbar^g \ip{t^{k_1} \alpha_1, \ldots, t^{k_n} \alpha_n }^{\F, \Lag}_{g,n}  =\left( \frac{\partial}{\partial (t^{k_1} \alpha_1 ) } \cdots \frac{\partial}{\partial (t^{k_1} \alpha_1 ) } \hbar \log \Phi_{\Lag} (Z_{\F}) \right) (0) \in \C[[\hbar]].
$$
\end{definition}
These correlators depend, of course, on $\Lag$, but in a controlled way.  If we change polarization $\Lag$ to $\Lag'$, then the isomorphism
$$
\Phi_{\Lag} : \mc{F} ( H^\ast \mc{S}_+(X) ) \iso \Oo ( H^\ast ( \PV(X) ) [[t]] [2] ) [[\hbar]]
$$
changes by an isomorphism
$$
\Psi_{\Lag \to \Lag' } : \Oo ( H^\ast ( \PV(X) ) [[t]] [2] ) [[\hbar]]
$$
of the type considered by Givental \cite{Giv04}.  In particular, the dependence of these correlators on $\Lag$ is of a polynomial nature. 

\subsection{}
In section \ref{section quantization} we gave a different definition of correlators, defined by considering the scale infinity effective interaction $\F[\infty]$ on the harmonic forms.  These correlators are simply a special case of the correlators considered here, corresponding to a particular choice of polarization $\Lag$.  

Indeed, let $g$ be a K\"ahler metric on $X$.  Let
$$
\mc{H} \mc{S}(X) \subset \mc{S}(X) = \PV(X)((t))[2]
$$
be the subspace of harmonic elements, defined using this metric.  The operators $\partial$ and $\dbar$ are zero on $\mc{H}(\mc{S}(X))$, so that one has  natural isomorphisms
$$
H^\ast \mc{S}( X) \iso \mc{H} \mc{S}(X)  \iso \mc{H} (\PV(X) ) ((t))[2] \iso H^\ast (\PV(X) ) ((t)) [2].
$$
Here, $\mc{H}(\PV(X))$ denotes the space of harmonic polyvector fields.  

Let 
$$
\Lag_g \subset H^\ast \mc{S} (X) 
$$
denote the subspace corresponding to 
$$t^{-1} \mc{H} ( \PV(X)) [t^{-1} ] [2]\subset \mc{H} (\PV(X) ) ((t)) [2].$$
\begin{lemma}
The correlation functions defined in section \ref{section quantization} are the correlation functions corresponding to this Lagrangian subspace $\Lag_g$. 
\label{lemma lagrangian metric}
\end{lemma}
\begin{proof}
This is immediate. 
\end{proof}

\subsection{}
It turns out that the polarization $\Lag_g$ of $\mc{S}(X)$ does not, in fact, depend on the metric $g$.  However, it \emph{does not} vary holomorphically with $X$.  This leads to the famous ``holomorphic anomaly'' of \cite{BerCecOog94}.  

In order to see this, let us reinterpret the symplectic vector space $H^\ast \mc{S}(X)$, and the polarization $\Lag_g$, in terms of Hodge theory.   
\begin{proposition}
There is a natural isomorphism of symplectic vector spaces
$$
H^\ast \mc{S}(X) \iso H^\ast(X) ((t)) [d]
$$
where the right hand side is equipped with the symplectic pairing
$$
\omega(\alpha f(t), \beta g(t) = \int_{X} \alpha \wedge \beta \op{Res} \left( t^{2-d} f(t) g(-t) \d t \right). 
$$
This isomorphisms is $\C((t))$-linear. 

Further, under this isomorphism, the subspace
$$
H^\ast \mc{S}_+(X) \subset H^\ast \mc{S}(X)
$$
corresponds to the subspace spanned by the subspaces
$$
F^p H^\ast(X) \otimes t^{d - p  -1} \C[[t]] \subset H^\ast(X)((t))
$$
for each $p \ge 0$.  Here $F^p H^\ast (X) \subset H^\ast(X)$ is the $p^{th}$ piece of the Hodge filtration.
\end{proposition}
Thus, $H^\ast(S_+(X))$ is spanned by elements of the form $t^n \alpha$ where $\alpha \in F^p H^\ast(X)$ and $n \ge d-p-1$. 
\begin{proof}
Recall that $\mc{S}(X)$ is $\PV(X) ((t)) [2]$, with differential $\dbar - t \partial$.  The holomorphic volume form on $X$ gives an isomorphism
$$
\PV^{i,j}(X) \to \Omega^{d - i, j}(X)
$$
by sending
$$
\alpha \to \alpha \vee \Omega,
$$
where $\Omega \in \Omega^{d,0}(X)$ is the holomorphic volume form.  

Under this isomorphism, the operators $\dbar$ and $\partial$ on $\PV^{i,j}(X)$ correspond to the standard operators $\dbar$ and $\partial$ on $\Omega^{\ast,\ast}(X)$.

Let us define a map
\begin{align*}
\Gamma : \PV^{i,j}(X) & \to \Omega^{d-i,j} (X) ((t))   \\ 
\Gamma(\alpha) &= (-t)^{i-1} (\alpha \vee \Omega). 
\end{align*} 
We can extend $\Gamma$, by $\C((t))$-linearity, to an isomorphism
$$
\Gamma : \PV(X) ((t))[2] \to \Omega(X) ((t)) [d] .
$$
Note that 
\begin{align*}
\Gamma ( \dbar \alpha ) & = \dbar \Gamma ( \alpha ) \\ 
\Gamma (-t  \partial \alpha ) &= \partial \Gamma(\alpha). 
\end{align*}
Thus, $\Gamma$ is a cochain isomorphism, and induces the desired isomorphism
$$
H^\ast \mc{S}(X) \iso H^\ast (X) ((t)) [d].
$$
The remaining properties are straightforward to verify. 
\end{proof}

\begin{lemma}
Polarizations of $H^\ast \mc{S}_+(X)$ are in bijection with splittings of the Hodge filtration on $H^\ast (X)$.
\end{lemma}
\begin{proof}
A splitting of the Hodge filtration is described by a filtration
$$
\til{F}^p H^\ast (X) \subset  \til{F}^{p+1} H^\ast(X) \cdots
$$
for $p \ge 0$, such that
$$
\til{F}^p H^\ast(X)  \oplus F^p H^\ast(X) = H^\ast(X).
$$
Given such a splitting, we define the corresponding Lagrangian subspace $\Lag \subset H^\ast (X) ((t))$ to be the subspace spanned by the spaces
$$
\til{F}^p H^\ast (X) \otimes t^{d - p - 1} \C[t^{-1} ] \subset H^\ast(X) ((t)),
$$
for $p \ge 0$.  Thus, if $\alpha \in H^\ast(X)$, $t^n \alpha \in \Lag$ if $\alpha \in \til{F}^p H^\ast(X)$ and $n \le d-p-1$. 

We can transfer this subspace to subspace of $H^\ast \mc{S}(X)$ using the isomorphism $H^\ast \mc{S}(X) \iso H^\ast(X)((t))$, to obtain a polarization of $H^\ast \mc{S}(X)$. 

Conversely, the subspace $\til{F}^p H^\ast (X)$ is defined as follows.  Let
$$
\pi_k  : H^\ast(X)((t)) \to H^\ast(X)
$$
be the projection onto the coefficient of $t^k$.  Then, we let
$$
\til{F}^p H^\ast(X) = \pi_{d - p - 1} \Lag \subset H^\ast(X).  
$$
The fact that $\Lag$ is closed under multiplication by $t^{-1}$ implies that $\til{F}^p H^\ast(X)$ is an increasing filtration.  It is straightforward to verify that this filtration is complementary to the Hodge filtration. 
\end{proof}
\begin{definition}
The \emph{complex-conjugate splitting} of the Hodge filtration is given by the subspaces
$$
\til{F}^p H^n (X) = \br{F}^{n-p} H^n(X),
$$
where $\br{F}^{n-p}$ is the subspace complex conjugate to $F^{n-p} H^n(X)$.

We will let
$$
\Lag_{\br{X} } \subset H^\ast \mc{S}(X)
$$
denote the polarization corresponding to the complex-conjugate splitting.  
\end{definition}
\begin{lemma}
Let $g$ be a K\"ahler metric on $X$, and let $\Lag_g \subset H^\ast \mc{S}(X)$ be defined as in \ref{lemma lagrangian metric}.  Then,  $\Lag_g = \Lag_{\br{X}}$. \end{lemma}
\begin{proof}
We defined $\Lag_g$ by using the metric $g$ to identify
$$
H^\ast\mc{S}(X) = \mc{H} (\PV(X)) ((t))
$$
where $\mc{H}(\PV(X)) $ is the space of harmonic polyvector fields.    The subspace $\Lag_g$ was defined to be simply $t^{-1} \mc{H}(\PV(X)) [t^{-1}]$. 

Under the isomorphism between polyvector fields and the de Rham algebra of $X$, harmonic polyvector fields correspond to harmonic forms.    It follows that the splitting of the Hodge filtration corresponding to $\Lag_g$ is the one where, when we identify the cohomology $H^\ast(X)$ with the space $\mc{H}(X)$ of harmonic forms, we set
$$
\til{F}^p H^\ast(X) = \oplus_{i \ge p} \mc{H}^{n-i,i} (X) .
$$
(Here $\mc{H}^{p,q}(X)$ denotes harmonic elements of $\Omega^{p,q}(X)$).  It is standard that this is the complex-conjugate splitting.  
\end{proof}

\begin{corollary}
The correlation functions defined in section \ref{section quantization} depend only on the homotopy class of the quantization $(\{\F[L]\}, g)\in \op{Quant}(X)$.
\end{corollary}
\begin{proof}
By lemma \ref{lemma lagrangian metric}, these correlation functions correspond to the choice $\Lag_{\br{X}}$ of polarization.  For a fixed polarization, the correlation functions only depend on the homotopy class of the quantization.  
\end{proof}

\subsection{}
\begin{definition}
Let $X$ be a manifold of dimension $d$. For any complex structure $J$ on $X$, let $F_J^k H^\ast (X)$ denote the Hodge filtration on $H^\ast(X)$ arising from $J$. 

Two complex structures $J, J'$ on $X$ are \emph{complementary} if
$$
F_J^p H^n(X) \oplus F^{n-p}_{J'} H^n (X) = H^n (X).
$$
\end{definition}
Note that for any complex structure $J$ on $X$, $J$ and $-J$ are complementary.  Note also that the filtration $F_J^k H^\ast(X)$ defined by $J$ only depends on the class of $J$ in the Teichmuller space of complex structures on $X$ (defined to be the space of complex structures divided by the connected component of the identity in $\op{Diff}(X)$). 

Now let $X$ denote a Calabi-Yau manifold, and let $Y$ be a complex manifold with a homotopy class of diffeomorphism $\phi : Y \iso X$.   Let us suppose that the complex structures on $X$ and $Y$ are complementary.   Let $\F$ be a quantization of the BCOV theory on $X$.  Then, we can define correlators depending on both $X$ and $Y$ by using the splitting of the Hodge filtration on $X$ given by $Y$.

\section{The classical action functional as a Lagrangian cone}
\label{section_cone}
We have seen how to translate between action functionals satisfying the classical master equation and Lagrangian cones on $\PV(X)((t))$.  It is natural to ask whether the Lagrangian cone associated to our classical BCOV action functional has a more geometric description.  In this section we will show that this is indeed the case.

\subsection{}
Consider the dg symplectic vector space $\PV(X)((t))$ with differential $\dbar - t \partial$, as before.  Let us define a formal graded submanifold $\mc{L} \subset \PV(X)((t))$, based at the origin $0 \in \PV(X)((t))$, by saying that
$$
\mc{L} = \left\{ t \left( 1- e^{g/ t } \right) \mid g \in \PV(X)[[t]] \right\}.
$$
This presentation is a little informal.  More precisely, we can define this formal submanifold via its functor of points.  This is a functor from nilpotent Artinian graded algebras $R$ (with maximal ideal $m \subset R$) to sets.  If $R$ is an Artinian graded algebra, then the $R$-points of $\PV(X)((t))$ is the set of degree $0$ elements of $\PV(X)((t)) \otimes m$.  We define $\mc{L}(R)$ to be the set of those $\alpha \in \PV(X)((t))\otimes R$ which are of degree $0$, and which can be expressed (necessarily in a unique way) in the form 
$$
\alpha = t \left ( 1 - e^{g / t }\right)
$$
for some $g \in \PV(X)[[t]] \otimes m$.  This expression makes sense because the maximal ideal $m \subset R$ is nilpotent. 

\subsection{}
Let us now list some fundamental properties of $\mc{L}$.
\begin{lemma}
The submanifold $\mc{L} \subset \PV(X)((t))$ is a Lagrangian submanifold, preserved by the differential on $\PV(X)((t))$.  
\end{lemma}
\begin{proof}
First we will check that $\mc{L}$ is Lagrangian.  As before, let $(R,m)$ be an Artinian graded algebra, and let $g \in \PV(X)[[t]] \otimes m$ be a degree $0$ element.  Let $\alpha = t (e^{g/t} - 1)$.  Note that the tangent space to $\mc{L}$ at $\alpha$ is
$$
T_\alpha \mc{L} = e^{g/ t } \PV(X)[[t]] \subset \PV(X)((t)) \otimes m.
$$
Let us expand $g$ as 
$$
g = \sum t^i g_i
$$
where $g_i \in \PV(X)[[t]] \otimes m$.  Note that 
$$
e^{g/t} \PV(X)[[t]] = e^{g_0/t} \PV(X)[[t]].
$$
Finally, to verify that this is a Lagrangian subspace of $\PV(X)((t))$, we need to check that multiplying by $e^{g_0 / t}$ is a symplectomorphism. This amounts to verifying that multiplying by $g_0 / t$ is an infinitesimal symplectomorphism, which is immediate. 

Next, let us check that $\mc{L}$ is preserved by the differential.  We view the differential $\dbar - t \partial$ as a vector field on $\PV(X)((t))$; we need to check that this vector field, when restricted to $\mc{L}$, is tangent to $\mc{L}$. 

As before, let
$$
\alpha = t ( 1 - e^{g/t} ) \in \mc{L}
$$
for some $g \in \PV(X)[[t]] \otimes m$.   We need to verify that
$$
(\dbar - t \partial) \alpha \in T_\alpha \mc{L}= e^{g/t} \PV(X)[[t]].
$$

Then, 
$$
(\dbar - t \partial) \alpha =  -e^{g/t} \dbar g + t^2 \partial (e^{g/t}). 
$$
Recall that
$$
\partial (e^{g/t} ) = e^{g/t} (t^{-1} \partial g + t^{-2} \{g,g\} ).
$$
Thus, $(\dbar - t \partial) \alpha$ is in $T_\alpha \mc{L}$ as desired. 
\end{proof}
 
\subsection{}
Note that the projection
$$
\pi : \mc{L} \to \PV(X)[[t]]
$$
(defined using the polarization $\PV(X)((t)) = \PV(X)[[t]] \oplus t^{-1} \PV(X)[t^{-1} ]$)
is an isomorphism of formal graded manifolds. 

It follows that there is a functional
$$
F_{\mc{L}} \in \Oo(\PV(X)[[t]])
$$
such that 
$$
\mc{L} = \op{Graph} (\d F_{\mc{L}} ).
$$
More precisely, for all $f \in \PV(X)[[t]]$, and $g \in \PV(X)[[t]] \otimes m$, we have
$$
\frac{\partial \F_{\mc{L}}} {\partial f}  (g) = \Omega (f,\pi^{-1} (g ) ). 
$$

Our main theorem in this section is the following.
\begin{theorem}
$F_{\mc{L}}$ is the classical BCOV action functional $I$.  
\end{theorem}
Before we prove this theorem, it will be helpful to have a better understanding of the map $\pi^{-1}$.

Let us parametrize $\mc{L}$ as follows.  If $f \in \PV(X)$, and $g \in \PV(X)[[t]]$, then let
$$
B(f,g) = t - t e^{f/t} \left( 1 + g \right) \in \mc{L}.
$$
Further, sending $(f,g)$ to $B(f,g)$ is an isomorpism from $\PV(X) \oplus \PV(X)[[t]]$ to $\mc{L}$. 

We will try to understand $\pi^{-1}$ in these coordinates.  Let $h \in \PV(X)[[t]]$.  To define $\pi^{-1}(h)$, we need to find $f \in \PV(X)$ and $g \in  \PV(X)[[t]]$ such that $\pi B (f, g ) = h$.  

Let us expand $h = \sum_{i \ge 0} h_i t^i$, and $g = \sum_{i \ge 0} g_i t^i$.  Then the equation $\pi B(f,g ) = h$ amounts to the following system of equations:
\begin{align*}
-h_0 &= f + f g_0 +  \tfrac{1}{2}  f^2 g_1 + \tfrac{1}{3!}  f^3 g_2 + \dots  \\
-h_1 &= g_0 + f g_1 + \tfrac{1}{2}  f^2 g_2 + \tfrac{1}{3!}  f^3 g_3 + \dots \\
-h_2 &= g_1 + f g_2 + \tfrac{1}{2}  f^2 g_3 + \tfrac{1}{3!}  f^3 g_4 + \dots
\end{align*}
(Recall that we are dealing with formal manifolds: these expressions will terminate if $f$ and $g$ are accompanied by auxiliary nilpotent parameters).

It is clear from these formulae that given $h$, there is a unique $(f,g)$ satsfying these equations. Indeed, to determine $f$ and $g$, consider the algebra isomorphism
\begin{align*}
\Phi : \C[[x_0,x_1,x_2,\dots ]] & \to \C[[z, y_0,y_1, \dots]] \\
\Phi ( x_0 ) &= -z - \sum_{i \ge 1} \frac{z^i}{i!} y_{i-1}\\  
\Phi ( x_j ) &= -\sum_{i \ge 0} \frac{z^i}{i!} y_{i+j - 1} \text{ if } j > 0
\end{align*}
Note that our system of equations above says that
$$
h_i = \Phi ( x_i ) ( f, g_0, g_1, \dots ). 
$$
In other words, we determine $h_i$ by substituting $f$ for $z$ and $g_i$ for $y_i$ in the formal series $\Phi(x_i) \in \C[[z, y_j ]]$. 

In a similar way, we can express $f$ and $g$ in terms of $h$ by using the inverse
$$
\Phi^{-1} : \C[[z, y_0, y_1, \dots]] \to \C[[x_0, x_1, \dots]].
$$
Indeed,
\begin{align*}
\Phi^{-1} ( z ) (h_0, h_1, \dots ) &= f \\
\Phi^{-1} ( y_j ) (h_0, h_1, \dots ) &= g_j. 
\end{align*}
Thus, we have proved the following lemma.
\begin{lemma}
There exists a unique collection of formal series $\Gamma$, $\Lambda_0, \Lambda_1, \dots$ in the algebra $\C[[x_0, x_1, \dots, ]]$ such that, for all $h \in \PV(X)[[t]]$, 
$$
\pi^{-1} (h) = B\left( \Gamma ( h_0, h_1, \dots ) , \sum t^i \Lambda_i ( h_0, h_1, \dots )   \right).
$$
\end{lemma}
\begin{proof}
In the notation above, $\Gamma = \Phi^{-1}(z)$ and $\Lambda_i = \Phi^{-1}( y_ i )$. 
\end{proof}

\subsection{}
Let us consider again the generating function $\F_{\mc{L}} \in \Oo(\PV(X)[[t]])$ for our Lagrangian cone.
\begin{lemma}
There exists a unique series of constants $C(k_1,\dots, k_n ) \in \Q$, for $k_1, \dots, k_n \in \Z_{\ge 0}$ and each $n \ge 3$, such that, for each $\alpha_1,\dots, \alpha_k \in \PV(X)$ and $k_1,\dots,k_n \in \Z_{\ge 0}$, we have
$$
\left( \dpa{(t^{k_1} \alpha_1 ) }  \dpa{(t^{k_2} \alpha_2 ) }  \dots \dpa{(t^{k_n} \alpha_n ) }  \F_{\mc{L}} \right) (0) = \op{Tr} (\alpha_1 \dots \alpha_n ) C(k_1, \dots, k_n) .
$$
\end{lemma} 
\begin{proof}
Consider the formal power series $\Gamma, \Lambda_0, \Lambda_1 \dots $ as in the previous lemma.  By definition of $\F_{\mc{L}}$, we must have, for each $\phi, \psi \in \PV(X)[[t]]$,
$$
\left( \dpa{\phi} \F_{\mc{L}} \right) (\psi) = \Omega (\phi,  \pi^{-1}(\psi) ).
$$
Now, we have seen that we can write $\pi^{-1}(\psi)$ as $B(f,g) \in \mc{L}$, where $f \in \PV(X)$ and $g = \sum g_i t^i$ are built from products of the terms $\phi_i$ in the expansion $\phi = \sum t^i \phi_i$.   Since the symplectic pairing on $\PV(X)((t))[2]$ is defined using the pairing $\op{Tr}(\alpha \beta)$ on $\PV(X)$,  the lemma follows immediately.
\end{proof}

Recall that we normally consider our symplectic vector space to be $\PV(X)((t))[2]$.  Let us view $\F_{\mc{L}}$ as an element of $\Oo(\PV(X)[[t]][2])$.
\begin{lemma}
$\F_{\mc{L}}$ is of cohomological degree $6-2d$, where $d = \dim_{\C} X$.
\end{lemma}
\begin{proof}
Let $G$ be the $\C^\times$ action on $\PV(X)((t))[2]$ induced by the cohomological grading.  Thus, $G(\lambda)(\alpha) = \lambda^k \alpha$ if $\alpha$ is of cohomological degree $k$.  Note that we can write $G(\lambda)$ as a product of $E(\lambda)$, defined by $E(\lambda)(\alpha) = \lambda^{-2} \alpha$, and $\til{G}(\lambda)$, which is the grading $\C^\times$ action on $\PV(X)((t))$.

We will remove the $[2]$ from the notation, and just write $\PV(X)((t))$.  The first thing to show is that
$$
\mc{L} = \{ t (1-e^{f/t} ) \mid f \in \PV(X)[[t]] \} \subset \PV(X)((t))
$$
is preserved by the action of $G(\lambda)$. 

Note that $\til{G}(\lambda)$ is an algebra automorphism of $\PV(X)((t))$. Now,
\begin{align*}
G(\lambda) \left( t (1-e^{f/t}  )  \right) &= \lambda^{-2} \til{G}(\lambda) \left( t (1-e^{f/t}  )  \right)\\
&= t - t \exp \left(t^{-1} \lambda^{-2} \til{G}(\lambda)(f)   \right) . 
\end{align*}
The second equality uses the fact that $\til{G}(\lambda)$ is an algebra homomorphism and that $\til{G}(\lambda)(t) = \lambda^2 t$.  

Now, if $f \in \PV(X)[[t]]$ then so is $\lambda^{-2} \til{G}(\lambda)(f)$, and so $G(\lambda)$ preserves the submanifold $\mc{L}$.

Note that the symplectic form $\Omega$ on $\PV(X)((t))[2]$ is of cohomological degree $6 - 2 d$.  Thus, for $\mc{L}$ to be invariant under the action of $G(\lambda)$, the generating function $\F_{\mc{L}}$ must be of cohomological degree $6-2d$.   
\end{proof}

\begin{lemma}
$\F_{\mc{L}}$ satisfies the classical master equation, the dilaton equation, and the string equation.
\end{lemma}
\begin{proof}
The classical master equation is equivalent to the statement that the submanifold $\mc{L} \subset \PV(X)((t))$ is preserved by the differential, and we have already seen that this is the case.  

According to Givental \cite{Giv01}, the dilaton equation has the following geometric interpretation. Let
$$
\mc{L}' = \mc{L} - t \subset \PV(X)((t))
$$
be the formal submanifold $\mc{L}$, translated so that it is defined near $t$.  The dilaton equation is then the statement that $\mc{L}'$ is a cone.

In our context, this is clear: after all, 
$$
\mc{L}' = \{-t e^{f/ t } \mid f \in \PV(X)[[t]]\}.
$$
Clearly, this is conic: if $\lambda \in \C$, then  
$$
- e^\lambda t e^{f/t} = -t \exp \left( t^{-1} (f + t \lambda)\right).
$$

Givental has also provided a geometric interpretation of the string equation.  The string equation is the statement that the cone $\mc{L}'$ is preserved by the infinitesimal symplectomorphism of $\PV(X)((t))$ given by multiplying by $t^{-1}$.  Again, it is clear that $\mc{L}'$ satisfies this property: if $\eps$ is a parameter of square $0$, we have
$$
(1+ \eps t^{-1} )\left(- t e^{f/t}\right) = - t \exp \left( t^{-1} (f + \eps)\right).
$$
\end{proof}

\begin{lemma}
For $\alpha,\beta,\gamma \in \PV(X)$, we have
$$
\left(\dpa{\alpha} \dpa{\beta}\dpa{\gamma} \F_{\mc{L}}\right)(0) =  \op{Tr}(\alpha \beta \gamma ).
$$
\end{lemma}
\begin{proof}
By definition, we have 
$$
\dpa{\alpha}\F_{\mc{L}} (\psi)  = \Omega(\alpha, \pi^{-1} \psi)
$$
for all $\alpha,\psi \in \PV(X)[[t]]$.  Let us take $\alpha$ and $\psi$ to be in $\PV(X)$.  Then,
$$
\pi^{-1}(\psi) = t ( 1 - e^{-\psi/t}) = t ( \psi/t -\psi^2 /2  t^2 + \dots ) .
$$
Thus, 
$$
\dpa{\alpha}\F_{\mc{L}} (\psi) = -\tfrac{1}{2} \Omega(\alpha, \psi^2 / t ).
$$
Since the symplectic form is defined by 
$$
\Omega(\alpha f(t) , \beta g(t) ) = \Tr(\alpha \beta) \op{Res}(f(t) g(-t) \d t) 
$$
we see that
$$
\dpa{\alpha}\F_{\mc{L}} (\psi) = \tfrac{1}{2} \Tr (\alpha \psi^2).
$$
Taking derivaties with respect to $\psi$ completes the proof. 
\end{proof}

\begin{corollary}
The function $\F_{\mc{L}}$ is equal to the classical BCOV action functional.
\end{corollary}
\begin{proof}
We have seen that $\F_{\mc{L}}$ satisfies
$$
\left( \dpa{(t^{k_1} \alpha_1 ) }  \dpa{(t^{k_2} \alpha_2 ) }  \dots \dpa{(t^{k_n} \alpha_n ) }  \F_{\mc{L}} \right) (0) = \op{Tr} (\alpha_1 \dots \alpha_n ) C(k_1, \dots, k_n) 
$$
for some constants $C(k_1,\dots,k_n) \in \Q$.  The grading conditions on $\F_{\mc{L}}$ imply that $C(k_1,\dots,k_n) = 0$ unless $\sum k_i = n-3$.  The dilaton and string equation imply that $C(k_1,\dots,k_n)$ are uniquely determined by $C(0,0,0)$ which is $1$.  Since the classical BCOV action functional $I$ satisfies all the same properties, it must coincide with $\F_{\mc{L}}$. 
\end{proof}

\section{Quantization on an elliptic curve}
\label{section elliptic quantization}
In this section we will prove the following theorem.
\begin{theorem}
There exists a unique translation invariant quantization of the BCOV theory on an elliptic curve $E$ which satisfies the dilaton equation.

Further, this quantization also satisfies the string equation, and the remaining Virasoro constraints. 
\end{theorem}
The existence part of this result relies on results of \cite{CosLi12}. 

The Virasoro constraints are generalizations of the string and dilaton equation; however, they are much harder to write down, and unlike the string and dilaton equation they depend on the dimension of the variety.   The Virasoro constraints will be defined precisely later. 

\subsection{}
Let us start by outlining the proof of this theorem.  The first step is to analyze the obstruction-deformation complex controlling quantizations of BCOV theory on an elliptic curve $E$.

In \cite{Cos11}, an obstruction-deformation complex is constructed which controls quantizations of any perturbative quantum field theory.    As explained in \cite{Cos11}, Chapter 5, this obstruction-deformation complex is, as a graded vector space, the space $\Ool(\Fields)$ of local functionals on our space $\Fields = \PV(E)[[t]]$ of fields.   The differential on $\Ool(\Fields)$ is given by 
$$
Q + \{I, - \}
$$
where, as before, $Q = \dbar - t \partial$ and $I \in \Ool(\Fields)$ is the classical BCOV action functional. 

Our grading convention is such that $H^0$ of this complex describes deformations of a quantization, $H^1$ contains obstructions to quantizations, and $H^{-1}$ consists of automorphisms of a quantization.   Thus, if one has a quantization defined modulo $\hbar^{k+1}$, the obstruction to quantizing to the next order in $\hbar$ is an element of $H^1$ of this complex; and the space of possible quantizations to the next order is a torsor for $H^0$ of this complex.

If one thinks of the space of possible quantizations as a simplicial set, then at each order in $\hbar$, if the obstruction vanishes, the simplicial set of lifts to the next order is a torsor for the simplicial abelain group constructed by the Dold-Kan correspondence from the cochain complex $\Ool(\Fields)$.  The $i$'th homotopy group of this simplicial set is $H^{-i} (\Ool(\Fields))$.  

\subsection{}
We are interested in quantizations which satisfy the dilaton equation.  The dilaton equation is defined by the operator
$$
\partial_{Dil} = \partial_{t \cdot 1} - \op{Eu} : \Ool(\Fields) \to \Ool(\Fields),
$$
which commutes with the differential $Q + \{I,-\}$.   

The obstruction-deformation complex for quantizing to order $g$ in $\hbar$ will be the homotopical version of the $2g-2$ eigenspace of the operator $\partial_{Dil}$ on $\Ool(\Fields(E))$.  More precisely, this obstruction-deformation complex is 
$$\left( \Ool(\Fields) [\eps] ,  Q + \{I,-\} + \eps (\partial_{Dil} + 2 - 2 g) \right).$$
This obstruction-deformation complex further breaks up under Hodge weight: recall that the subspace
$$
t^m \Omega^{0,\ast} ( \wedge^k T X ) \subset \E(X)
$$
has Hodge weight $k + m - 1$.  
\begin{definition}
If $U \subset E$ is an open subset, let 
$$
\op{Obs}_{a,b}(U) \subset \left( \Ool(\Fields) [\eps] ,  Q + \{I,-\} + \eps (\partial_{Dil} + a \right)
$$
be the subcomplex of Hodge weight $b$. (The number $a$ appearing here will be referred to as the dilaton weight).

As $U$ varies, these complexes form a sheaf of cocahin complexes which we denote $\op{Obs}_{a,b}(U)$.  We will let $\op{Obs}^i_{a,b}$ denote the sheaf of sections of cohomological degree $i$, and we will let $\mc{H}^i (\op{Obs}_{a,b} )$ denote the cohomology sheaf of the sheaf of complexes $\op{Obs}_{a,b}$. (Recall that the cohomology sheaf is defined by sheafifying the presheaf which sends $U$ to $H^i (\op{Obs}_{a,b}(U))$). 
\end{definition}
Because $\F_g[L]$ is required to be in cohomological degree $4 - 4g$ and Hodge weight $2 -2 g$, we see that the group controlling obstructions to quantizations to order $g$ is $H^{5 - 4g}(\Obs_{2-2g, 2-2g}(E))$.  

The following theorem summarizes the main features of the sheaf of complexes $\Obs_{a,b}$. The proof of this theorem is presented in the Appendix.  
\begin{theorem}
\label{theorem obstruction vanishing elliptic}
The cohomology sheaves $\mc{H}^i(\Obs_{a,b})$ are constant sheaves.  Further, 
$$
\mc{H}^{ i + 2 b  }(\Obs_{a,b}) = 0
$$
if $a \le 0$ and $(i,b) \le (0,0)$ in the lexicographical ordering on $\Z \times \Z$. 
\end{theorem}
This implies that if there is a quantization, it is unique (even unique up to a contractible choice). However, there may be an obstruction to quantization; at genus $g$ the possible obstructions are sections of the sheaf $\mc{H}^{1 + 4 - 4g} (\Obs_{2-2g,2-2g})$.
The proof that all of these obstructions vanish is outsourced to \cite{CosLi12}.

\section{Correlation functions and the $\br{\tau} \to \infty$ limit}
\label{section elliptic correlators}
We are interested in the correlation functions for the BCOV theory on $E_\tau$.  Recall, as explained in section \ref{section givental}, that the corelation functions for the theory depend on the choice of a splitting of the Hodge filtration on $E_\tau$.

We will view the space $\mbb{H}$ as the Teichm\"uller space of complex structures on a fixed manifold $S^1 \times S^1$.   We will identify
$$
H^1( E_\tau ) = H^1 ( S^1 \times S^1) = \C \{a,b\},
$$
where $a,b$ are Poincar\'e dual to cycles $S^1 \times 1$, $1 \times S^1$.   For each $\tau \in \mbb{H}$, we get a Hodge filtration $F^1_\tau H^1(S^1 \times S^1)$.  In the basis given by $a$ and $b$, the Hodge filtration is simply
$$
F^1_{\tau} = \C ( a + b \tau ).
$$
If $\sigma \in \mbb{H}$, then the complex conjugate space $\br{F}^1_{\sigma}$ to $F^1_\tau$ defines a splitting of the Hodge filtration on $H^1(E_\tau)$.   Let
$$
\ip{-}_{g,n}^{E_\tau, \br{\sigma} } : H^\ast (\PV(E_\tau ))[[\tau]] ^{\otimes n} \to \C
$$
denote the correlation functions for the BCOV theory on $E_\tau$, defined using the splitting of the Hodge filtration given by $\sigma$.

\begin{lemma}
The correlation functions are $\ip{-}_{g,n}^{E_\tau, \br{\sigma} }$, viewed as functions of $\sigma$, are  polynomials in $1 / (\tau - \br{\sigma})$.
\end{lemma}
\begin{proof}
Note that the correlation functions are constructed by an entirely algebraic procedure from the complementary filtration to the Hodge filtration.  It follows that the correlators depend in an algebraic way on the complementary filtration; so that if we have an algebraically varying family of complementary filtrations, parameterized by some algebraic variety $Z$, then the correlators will be algebraic functions on $Z$.

Let us explain this more precisely.  Let $R$ be a finitely-generated commutative algebra.  An $R$-family of splittings of the Hodge filtration on $E$ is an isomorphism of $R$-modles
$$
H^1(E) \otimes R \iso R \oplus F^1 H^1(E)\otimes R,
$$
which splits the inclusion of $R$-modules
$$
F^1 H^(E) \otimes R \into H^1(E)  \otimes R.
$$
If we have such an $R$-family of splittings of the Hodge filtration, then we can consider the corresponding $R$-family of polarizations of the symplectic vector spaces $H^\ast(\mc{S}(E))$, and the resulting isomorphism of $\mc{W}(H^\ast(\mc{S}(E))) \otimes R$-modules
$$
\mc{F}( H^\ast(\mc{S}_+(E))) \otimes R \iso \Oo(H^\ast(\mc{S}_+(E)))[[\hbar]]  \otimes R.
$$
(One needs to be a little careful when defining the tensor product with $R$ here: the Weyl algebra and Fock module are pro-vector spaces.  Every element of the inverse system is tensored with $R$, so  that the Weyl algebra and Fock module become pro-$R$-modules).

Because they are defined by differentiating the partition function
$$
Z \in  \Oo(H^\ast(\mc{S}_+(E)))[[\hbar]]  \otimes R
$$
with respect to elements of $H^\ast(\mc{S}_+(E))$ and then evaluating at zero, it is clear that the correlators are elements of $R$.

It remains to verify that the splitting of the Hodge filtration on $H^1(E)$ defined by $\br{F}^1_\sigma$ depends in a polynomial way on $(\tau - \br{\sigma})^{-1}$.   Let $\alpha, \beta \in H^1(E)$ denote the basis Poincar\'e dual to the standard $a$- and $b$-cycles. Let
$$
\d z = \alpha + \tau \beta
$$
be the holomorphic one form for complex structure $\tau$, which spans $F^1_{\tau}H^1(E)$.  Note $\beta$ and $\d z$ are linearly independent, and that the subspace $\br{F}^1_{\sigma}( H^1(E))$ is spanned by
$$
(\tau - \br {\sigma} ) ^{-1} \d z + \beta.
$$
Clearly this is a family of complementary subspaces which depends in a polynomial way on $(\tau - \br{\sigma})^{-1}$.
\end{proof}
Note that this lemma implies, in particular, that the $\br{\sigma} \to i \infty$ limit of the correlators exist. We will use the notation
$$
\ip{-}_{g,n}^{E_\tau, \infty  } = \lim_{\br{\sigma} \to \infty} \ip{-}_{g,n}^{E_\tau, \br{\sigma}  }  : H^\ast (\PV(E_\tau ))[[\tau]] ^{\otimes n} \to \C.
$$
to denote the correlators correspondong to the limiting splitting of the Hodge filtration $\lim_{\br{\sigma} \to \infty} \br{F}^1_{\sigma}$. Physicists often write formulae like
$$
\lim_{\br{\tau} \to \infty} \ip{-}_{g,n}^{E_\tau, \br{\tau}}
$$
to denote these correlators.    We will sometimes abuse notation in this way.

\begin{remark}
Although we only discuss the example of elliptic curves, the
$\bar\tau\to \infty$ actually makes sense on general Calabi-Yau
moduli spaces. In that case we need to specify a large complex
limit, and $\bar\tau\to \infty$ simply replaces the complex
conjugate Hodge filtration by the limiting monodromy filtration.
The mirror symmetry conjecture states that the higher genus B-model
correlation functions is equivalent to the higher genus
Gromov-Witten theory on the mirror under $\bar\tau\to \infty$
\cite{BerCecOog94}.
\end{remark}

\section{Holomorphic properties of correlation functions}
\label{section holomorphic correlators}
In this section we will show the following.
\begin{theorem}
Let $\alpha_1, \ldots, \alpha_n$ be holomorphic sections of the bundle on $\mbb{H}$ whose fibre at $\tau$ is $H^\ast (\PV(E_\tau))$.  Then, the correlation functions
$$
\ip{\alpha_1 t^{k_1}, \ldots, \alpha_n t^{k_n} }_{g,n}^{E_\tau, \br{\sigma} } \in \C
$$
are holomorphic functions of $\tau$.
\label{theorem holomorphic correlators}
\end{theorem}

Now, the correlation functions are constructed directly from the partition function $Z(E)$ of the BCOV theory.  Recall from section \ref{section givental}  that the partition function $Z(E)$ of the BCOV theory is a class in the Fock space
$$
\mc{F} ( H^\ast \mc{S}_+(X) ) [\hbar^{-1} ]
$$
where
\begin{align*}
\mc{S}(E) & = \PV(E) ((t)) [2 ] \\
\mc{S}_+(E) &= \PV(E) [[t]] [2]
\end{align*}
(both of these cochain complexes are equipped with the differential $\dbar - t \partial$).    As explained in section \ref{section givental}, the notation $\mc{F}( H^\ast \mc{S}_+(X))$ indicates the Fock space for the symplectic vector space $H^\ast \mc{S}(X)$ based on the Lagrangian subspace $H^\ast \mc{S}_+(X)$.

The Fock space $\mc{F} ( H^\ast \mc{S}_+(E)) $ defines, as $E$ varies, a holomorphic bundle on the moduli space of elliptic curves.

In order to verify that the correlation functions vary holomorphically with $\tau$, it suffices to prove the following.
\begin{proposition}
The partition function
$$
Z(E) \in \mc{F} ( H^\ast \mc{S}_+(E)) [\hbar^{-1} ]
$$
is a holomorphic section of the holomorphic bundle $\mc{F} ( H^\ast \mc{S}_+(E)) [\hbar^{-1} ]$ on the moduli space of elliptic curves.
\end{proposition}
\begin{proof}
It suffices to prove that $Z(E_{\tau})$ is a holomorphic section of this bundle on the upper half-plane $\mbb{H}$.     The partition function $Z(E_{\tau})$ is defined in terms of the quantization $\{\F[L]\}$ we have constructed of the BCOV theory.  Therefore, it suffices to verify that this quantization varies holomorphically with $\tau$.

A general formalism for quantum field theories over a wide class of Fr\'echet base rings is formulated in \cite{Cos11}.      Saying that our quantization of the BCOV theory depends holomorphically on $\tau$ amounts to saying that we can quantize the theory when we work over the base ring $\Omega^{0,\ast}(\mbb{H})$.

Let $\pi : E \to \mbb{H}$ be the universal elliptic curve over $\mbb{H}$.   Let $\PV(E, \pi)$ denote the Dolbeaut resolution of the sheaf of relative polyvector fields on $E$.  That is, if $T \pi$ denotes the relative tangent bundle,
$$
\PV(E, \pi ) = \Omega^{0,\ast}(E) \oplus \Omega^{0,\ast} ( E, T \pi ) [-1].
$$
The fields of the relative version of the BCOV theory on $E$ are the $\Omega^{0,\ast}(\mbb{H})$-module $\PV(E, \pi ) [[t]] [2]$, with differential $\dbar - t \partial$.   The action functional (valued in $\Omega^{0,\ast}(\mbb{H})$) is defined exactly as before.

For each $\tau \in \mbb{H}$, let $\Obs_{a,b} (E_\tau)$ be the obstruction complex for quantizing the BCOV theory on $E_\tau$, as defined in section \ref{section elliptic quantization}.  The general theory of \cite{Cos11} implies that the obstruction group for quantizing the family version of the BCOV theory is defined by an $\Omega^{0,\ast}(\mbb{H})$ module which we denote $\Obs_{a,b}(E)$.  The fibres of this, at $\tau \in \mbb{H}$, are the obstruction groups $\Obs_{a,b}(E_\tau)$.

It follows that the cohomology $H^\ast ( \Obs_{a,b} (E_\tau ))$ forms a holomorphic bundle on $\mbb{H}$, and that there is a spectral sequence
$$
H^\ast_{\dbar} (\mbb{H}, H^\ast(\Obs_{a,b}(E_\tau ) ) \Rightarrow H^\ast ( \Obs_{a,b} (E ) ).
$$
Now, each $H^k(\Obs_{a,b}(E_\tau ))$ is a finite-rank holomorphic vector bundle on $\mbb{H}$, and therefore trivial.   It follows that
$$
H^k (\Obs_{a,b} (E)) = \Gamma ( \mbb{H}, H^k (\Obs_{a,b}(E_\tau ) ) ),
$$
where $\Gamma$ indicates the bundle of holomorphic sections.

The obstruction to quantizing the family BCOV theory over $\mbb{H}$, to genus $g$, is therefore a holomorphic section of the bundle $H^{5-4g}  (\Obs_{2-2g, 2-2g } (E_\tau ) )$ on $\mbb{H}$.   Since the obstruction vanishes for each $\tau \in \mbb{H}$, this holomorphic section vanishes.

\end{proof}

\section{Modular properties}
\label{section modular correlation}
Let
$$
\omega(\tau) \in H^1( E_\tau, T E_\tau )
$$
be the unique element with
$$
\op{Tr} (\omega) = \int_E \d z (\omega \vee \d z ) = 1.
$$
(Here $\d z$ denotes the holomorphic volume form on $E_\tau$ which pulls back to $\d z$ on $\C$).

Note that, because the trace map
$$
\op{Tr} : H^1(E_\tau, T E_\tau ) \to \C
$$
is a map of holomorphic line bundles on $\mbb{H}$, $\omega(\tau)$ defines a holomorphic section of the line bundle $H^1(E_\tau, T E_\tau)$ on $\mbb{H}$.  Futher, $\omega(\tau)$ is $SL_2(\Z)$-equivariant.

In this section we will prove the following.
\begin{theorem}
The correlators
$$
\ip{1 \cdot t^{k_1}, \ldots, 1 \cdot t^{k_m}, \omega \cdot t^{l_1}, \ldots, \omega \cdot t^{l_n} }_{g,n+m}^{E_\tau, \br{\sigma} }
$$
are modular functions on $\mbb{H} \times \mbb{H}$ (using the diagonal $SL_2(\Z)$-action) of weight $(2g-2+2n,0)$.  In other words, if
$$
A = \left( \begin{array}{c c}
a & b\\
c & d
\end{array} \right) \in SL_2(\Z),
$$
then
\begin{multline*}
\ip{1 \cdot t^{k_1}, \ldots, 1 \cdot t^{k_m}, \omega \cdot t^{l_1}, \ldots, \omega \cdot t^{l_n} }_{g,n+m}^{E_{A (\tau)}, \br{A(\sigma)} } \\ = (c \tau +d)^{2g-2+2n} \ip{1 \cdot t^{k_1}, \ldots, 1 \cdot t^{k_m}, \omega \cdot t^{l_1}, \ldots, \omega \cdot t^{l_n} }_{g,n+m}^{E_\tau, \br{\sigma} }.
\end{multline*}
\label{theorem modular correlators}
\end{theorem}

Note that these correlators $\ip{-}_{g,n+m}^{E_\tau, \br{\sigma}}$ only depend on the elliptic curve $E$, the holomorphic volume form $\Omega = \d z$ on $E$, and the splitting of the Hodge filtration on $H^1(E)$ coming from $\br{F}^1_\sigma$.      The modular properties stated in the theorem say that these correlation functions are functions on the moduli space elliptic curves $E$, with holomorphic volume form $\Omega$ and a splitting of the Hodge filtration, which are homogeneous of degree $2g-2+2n$ under rescaling of $\Omega$.

More precisely, let $\ip{-}_{g,n+m}^{E, \Omega, \br{F} }$ refer to the correlaton functions above associated to an elliptic curve $E$, volume form $\Omega$ and splitting of the Hodge filtration $\br{F}$.  Then the theorem is equivalent to the statement that
\begin{multline*}
\ip{1 \cdot t^{k_1}, \ldots, 1 \cdot t^{k_m}, \omega \cdot t^{l_1}, \ldots, \omega \cdot t^{l_n} }_{g,n+m}^{E, \lambda \Omega, \br{F} } \\  = \lambda^{2-2g - 2n} \ip{1 \cdot t^{k_1}, \ldots, 1 \cdot t^{k_m}, \omega \cdot t^{l_1}, \ldots, \omega \cdot t^{l_n} }_{g,n+m}^{E, \Omega, \br{F}  } .
\end{multline*}

Note that if we change $\Omega$ to $\lambda \Omega$, then $\omega \in H^1(E, T E)$ changes to $\lambda^{-2} \omega$.    The $2 n$ which appears in the scaling factor $\lambda^{2-2g-2n}$ can be attributed to this.

Thus, we need to show that the correlators associated to $(E,\Omega, \br{F})$, when viewed as linear maps
$$
\ip{-}_{g,n}^{E,\Omega,\br{F}} : \left( H^\ast(\PV(E)) [[t]] \right)^{\otimes n} \to \C
$$
change by $\lambda^{2-2g}$ when $\Omega$ is replaced by $\lambda \Omega$.

\subsection{}
Since these correlation functions are constructed directly from the quantization $\{\F^{(E,\Omega)}[L]\}$ of the BCOV theory on $E$ with volume form $\Omega$.  Thus, it suffices to show a similar scaling property for the BCOV theory on $(E,\Omega)$, when $\Omega$ is replaced by $\lambda \Omega$.

Thus, let us consider the BCOV theory on a fixed elliptic curve with varying holomorphic volume form $\Omega$.   We will fix a flat metric on $E$, independent of $\Omega$, and always use the gauge fixing operator $\dbar^\ast$ for this metric.  As explained in section \ref{section quantization}, the simplicial set of quantizations is fibred over that of gauge fixing conditions, so we can freely turn a quantization for one gauge fixing condition into a quantization for any other.

Let $\F^{E,\Omega} = \{\F^{E,\Omega}[L]\}$ denote the unique quantization of the BCOV theory on $(E,\Omega)$, using this fixed gauge fixing condition.
\begin{lemma}
Define
$$
R_\lambda (\F^{E,\Omega} ) [L] = \sum \hbar^g \lambda^{2g  - 2} \F^{E,\lambda \Omega} [L]
$$
for $\lambda \in \C^\times$.  Then, $R_\lambda (\F^{E,\Omega} )$ is a quantization of the BCOV theory on $(E,\Omega)$.
\end{lemma}
\begin{proof}
As we change $\Omega$ to $\lambda \Omega$, the heat kernel $K_L$ changes to $\lambda^{-2} K_L$.  Simiarly, the propagator $P(\eps,L)$ changes to $\lambda^{-2} P(\eps,L)$.

We want to check that $R_\lambda (\F^{E,\Omega} ) [L]$ satisfies the renormalization group equation and quantum master equation.   The renormalization group equation is the statement that
$$
\exp ( \hbar \partial_{P(\eps,L)} ) \exp ( R_\lambda (\F^{E,\Omega} ) [\eps] / \hbar ) = \exp ( R_\lambda (\F^{E,\Omega} ) [L] / \hbar).
$$
Note that $R_\lambda (\F^{E,\Omega} ) [L] /\hbar$ is obtained from $\F^{E, \lambda \Omega}$ by replacing $\hbar$ by $\lambda \hbar$.  The renormalization group equation for $\F^{E,\lambda \Omega}$ says that
$$
\exp ( \hbar \lambda^{-2} \partial_{P(\eps,L)} ) \exp ( \F^{E,\lambda \Omega}  [\eps] / \hbar ) = \exp (\F^{E,\lambda \Omega}  [L] / \hbar).
$$
By replacing $\hbar$ by $\lambda^2 \hbar$, we find the renormalization group equation for $R_\lambda (\F^{E,\Omega} [L] )$.

A similar argument proves the quantum master equation.  The locality axiom is automatic, as are the Hodge weight, cohomological degree and dilaton axioms.    It remains to verify that
$$
\F_0^{(E,\lambda \Omega)}[L] = \lambda^{2} F_0^{(E,\Omega)} [L]
$$
for all $L$.  Now, by the RG equation, it suffices to verify this when $L = 0$.  It is immediate from the definition of the classical BCOV action $I$ that changing $\Omega$ to $\lambda \Omega$ changes $I$ to $\lambda^2 I$.
\end{proof}

\begin{corollary}
For any splitting $\br{F}$ of the Hodge filtration on $H^1(E)$, the correlators
$$
\ip{-}_{g,n}^{E,\Omega,\br{F} } : \left( H^\ast (\PV(E))[[t]] \right)^{\otimes n} \to \C
$$
satisfy
$$
\lambda^{2g-2} \ip{-}_{g,n}^{E, \lambda \Omega,\br{F} } =  \ip{-}_{g,n}^{E,\Omega,\br{F} } .
$$
\end{corollary}
\begin{proof}
Indeed, the correlators $\lambda^{2g-2} \ip{-}_{g,n}^{E,\lambda \Omega,\br{F} } $ are those for the quantization $R_\lambda (\F^{E,\Omega})$ constructed in the previous lemma.  Since any two quantizations of the BCOV theory on $E$ are homotopic, and the correlators only depend on the homotopy class of the quantization, the result follows.
\end{proof}
This corollary completes the proof of theorem \ref{theorem modular correlators}.

Putting this result together with results of section \ref{section elliptic correlators} and section \ref{section holomorphic correlators} we see that the correlators
$$
\ip{1 \cdot t^{k_1}, \ldots, 1 \cdot t^{k_m}, \omega \cdot t^{l_1}, \ldots, \omega \cdot t^{l_n} }_{g,n+m}^{E_\tau, \br{\sigma} }
$$
are holomorphic in $\tau$, polynomail in $1 / (\tau - \br{\sigma})$, and modular of weight $2g-2+2n$.  Functions of this nature, which admit a $\tau \to i\infty$ limit, are called \emph{almost holomorphic modular forms}.

\section{Virasoro constraints}
\label{section virasoro}
In this section, we will prove that the $1$-dimensional BCOV theory satisfies the Virasoro constraints. These are mirror to the Virasoro constraints of Gromov-Witten invariants on elliptic curves, first discovered by \cite{Virasoro-GW}, and  proved by \cite{virasoro} in general.  In our setting, the Virasoro constraints will hold only up to homotopy, just like the string and dilaton equations hold only up to homotopy (as explained in section \ref{section quantization}).

Consider the following operators for each $m\geq 0$:
$$
       E_m:  \PV^{i,j}_X[[t]]\to \PV^{i,j}_X[[t]]\\
             t^k \alpha \to t^{m+k}\bracket{k+i}_{m+1} \alpha
$$
where $(n)_m=n(n+1)\cdots (n+m-1)$ is the Pochhammer symbol.   Let us define $E_{-1}$ by
$$
E_{-1} (t^k \alpha) = \begin{cases}
t^{k-1} \alpha & \text{ if } k > 0\\
0 & \text{ if } k = 0 
\end{cases} 
$$
The operators $E_m$ naturally induce operators acting on $\Oo(\E)$, which we denote by the same symbol. 

We will define Virasoro operators $\mc{L}_m[L]$  for $m \ge -1$, which act on the space $\Oo(\PV(E)[[t]])$ of functionals on the space of fields.  For $m \ge 0$, these operators will not depend on the scale $L$, whereas the operator $\mc{L}_{-1}[L]$ does depend on $L$. The operator $\mc{L}_{-1}[L]$ takes a little more work to define, so we will start by defining the operators $\mc{L}_m[L]$ for $m \ge 0$. 
\begin{definition} The operators $\mc{L}_m[L]$ are defined by 
$$      
\mc{L}_m[L] =-(m+1)!\dpa{(1\cdot t^{m+1})}+E_m.
$$
\end{definition}
If $m=-1$, the operator $\mc{L}_{-1}[L]$ will depend on the scale $L$.  This operator will be defined in terms of an auxiliary operator 
$$Y[L] : \PV(E)[[t]] \to \PV(E)[[t]],$$ 
defined by 
$$
Y[L](\alpha)=
\begin{cases}
 \int_0^L du \dbar^\ast \pa e^{-u[\dbar,\dbar^\ast]}\alpha & \alpha\in \PV(E)\\
 0 &\alpha\in t \PV(E)[[t]].
\end{cases}
$$
Recall that $I_3 \in \Oo(\PV(E)[[t]])$ is the local functional given by the order three component of the classical BCOV action.   This is defined by
$$
 \dpa{ \alpha \cdot t^k } \dpa{ \beta \cdot t^l } \dpa{ \gamma \cdot t^m }  I_3 = \delta_{k = 0} \delta_{l = 0} \delta_{m = 0} \Tr ( \alpha \beta \gamma).
$$

\begin{definition}
The Virasosro operator $ \mc{L}_{-1}[L]$ is defined by
$$
 \mc{L}_{-1}[L]=-\dpa{(1)}+E_{-1}-Y[L]+ \frac{1}{\hbar} \dpa{(1)}I_3.
$$
\end{definition}

Recall that our grading convention is that if $\alpha \in \PV^{i,j}(X)$, then $t^m \alpha$ is of cohomological degree $2m-2+i+j$ and Hodge weight $m-1+i$.  It follows that the operator $\mc{L}_m[L]$ is of cohomological degree $2m$ and Hodge weight $m$.
\begin{lemma}
The operators $\{\mc{L}_m[L] \mid m\geq -1\}$ satisfy the Virasoro relations
$$
\bbracket{\mc{L}_m[L],\mc{L}_n[L]}=(m-n)\mc{L}_{m+n}[L]
$$
for all $m,n\geq -1$.
\end{lemma}
\begin{proof}First we observe that, if $m,n \ge 0$,  
\begin{align*}
    \bbracket{E_m, E_n} t^k\alpha&=\bracket{(k+i)_{n+1}(k+i+n)_{m+1}-(k+i)_{m+1}(k+i+m)_{n+1} }t^{k+m+n}\alpha\\
    &=(n-m)(k+i)_{n+m+1}t^{k+m+n}\alpha\\
    &=(n-m)E_{n-m}t^k\alpha
\end{align*}
Therefore as operators acting on $\Oo(\E)$, we have
$$
       \bbracket{E_m, E_n}=(m-n)E_{m+n}
$$
On the other hand, 
$$
    \bbracket{\dpa{(1\cdot t^{m+1})}, E_n}=(m+1)_{n+1}\dpa{(1\cdot t^{m+n+1})}
$$
from which we can easily deduce the first equation. The other two are proved similarly.   The proof when one of $m$ or $n$ is $-1$ is similar. 
\end{proof}

\begin{lemma}
The operators $\{\mc{L}_m[L] \mid m \ge -1\}$ are compatible with the renormalization group flow,  quantum master equation and the dilaton operator in the following sense:
\begin{align*}
\exp\bracket{\hbar\partial_{P(\eps,L)}}\mathcal L_m[\epsilon]&=\mathcal L_m[L] \exp\bracket{\hbar\partial_{P(\eps,L)}} \\ 
\bbracket{\mathcal L_m[L], Q+\hbar \Delta_L}&=0\\\bbracket{\mathcal L_m[L], \partial_{Dil} -{2\hbar\dpa{\hbar}}}&=0.
\end{align*}
\end{lemma}
Recall that the dilaton operator $\partial_{Dil}$ is defined by
$$
\partial_{Dil} = \frac{\partial}{\partial (t \cdot 1 ) } - \op{Eu}.
$$
\begin{proof} 
This is a straightforward check.  
\end{proof}

\begin{lemma}
Let $I \in \Ool (\E)$ denote the the classical BCOV action.  Then,
$$
\mc{L}_m[0] e^{I / \hbar}  = e^{I/\hbar} 
$$
for all $m \ge -1$.  
\label{lemma virasoro classical}
\end{lemma}
\begin{proof}
Again, this is a straightforward calculation, 
\end{proof}

Suppose that $\{\F[L] \mid L \in \R_{> 0} \} \in \Oo(\E)[[\hbar]]$ is a quantization of the classical BCOV theory.  Let us define a family of functinals $(\mc{L}_m \F)[L]$ by the equation
$$
\F[L] + \delta (\mc{L}_m \F) [L]= \hbar \log \left( (1 + \delta \mc{L}_m[L]   ) \exp ( \F[L]/ \hbar ),  \right) 
$$
where $\delta$ is a parameter of cohomological degree $-2m$, Hodge weight $-m$ and square zero. 

Note that, if $m \ge 0$, 
$$
(\mc {L}_m \F)[L] = \mc{L}_m[L] \F[L].
$$
However,
$$
(\mc {L}_{-1} \F ) [L] =\left( \mc{L}_{-1}[L] - \hbar^{-1} \dpa{(1)} I_3 \right) \F[L] + \dpa{(1)} I_3 .
$$
\begin{corollary}
Suppose that $\{\F[L] \mid L \in \R_{> 0} \} \in \Oo(\E)[[\hbar]]$ is a quantization of the classical BCOV theory (as defined in section \ref{section quantization}).  

Then, so is the family of functionals
$$
\F[L] + \delta (\mc{L}_m \F)[L]
$$ 
where $\delta$ is a parameter of cohomological degree $-2m$, Hodge weight $-m$ and square zero. Further, if $\F[L]$ satisfies the dilaton equation, then so does $\F[L] + \delta (\mc{L}_m\F)[L]$.
\end{corollary}
\begin{proof}
For $m \ge 0$, this follows from Lemma \ref{lemma_derivation_rgflow} and Lemma \ref{lemma virasoro classical}.  For $m = -1$, the operator $\mc{L}_{-1}[L]$ is non-local for $L > 0$, so one needs in addition to verify the locality axiom.  This is an easy calculation using Feynman diagrams. 
\end{proof}
\begin{definition}
Let $\{\F[L]\}$ be a quantization of the classical BCOV theory.  Then we say $\{\F[L]\}$ satisfies the Virasoro constraint $\mc{L}_m$ if the deformation $\{\F[L] + \delta (\mc{L}_m \F)[L]\}$ is homotopically trivial. (Here $\delta$ is given the appropriate cohomological degree and Hodge weight).  

We say $\{\F[L]\}$ satisfies the Virasoro constraints if it satisfies the constraints $\mc{L}_m$ for each $m \ge -1$. 
\end{definition}
Just as with our discussion of the dilaton equation, to say that these first-order deformations of $\{\F[L]\}$ are homotopically trivial means (for $m \ge 0$) that there exists families of functionals $K_m[L] \in \hbar \Oo(\E)[[\hbar]]$ and $P_m[L] \in \hbar \Oo(\E)[[\hbar]]$ such that 
$$
    (\mc{L}_m \F)[L]= Q K_m[L]+\fbracket{\F[L], K_m[L]}_L+\hbar\Delta_L K_m[L] \label{Virasoro-L}\\
$$ 
The functionals $\{K_m[L]\}$ and $\{P_m[L]\}$ are required to satisfy a renormalization group equation and locality axiom similar to that satisfied by $\F_g[L]$. 

\subsection{}
One can analyze the obstruction to solving the Virasoro constraints order by order in $\hbar$. The obstruction-deformation complex for the BCOV theory on an elliptic curve is denoted by $\Obs_{a,b}(E)$ (in dilaton weight $a$ and Hodge weight $b$).     If we have a quantization $\{\F[L]\}$ of BCOV theory on an elliptic curve $E$ which satisfies all Virasoro constraints modulo $\hbar^g$, then the obstruction to satisfying the constraint $\mc{L}_m$, at genus $g$, is a class
$$
O_g(\mc{L}_m,E) \in H^{4-4g+m}(\Obs_{2-2g,2-2g+m} (E)).
$$
If all of these classes are zero, then the Virasoro constraints are satisfied. 

The main theorem of this section is that these classes do indeed vanish, so that the Virasoro constraints hold. In fact we will see that hte obstruction group vanishes.
\begin{theorem}
Let $\{\F[L]\}$ be our quantization of the BCOV theory on the elliptic curve $E$.  Then $\{\F[L]\}$ satisfies the Virasoro constraints. 
\end{theorem}
\begin{proof}
The proof of this theorem is by obstruction theory, and is very similar to the proof of the existence and uniqueness of the BCOV theory on the elliptic curve.

Recall that the obstruction complex $\Obs_{a,b}(E)$ on $E$ is the global sections of a complex of fine sheaves $\Obs_{a,b}$ on $E$.  It follows that we have a spectral sequence
$$
H^i (E, \mc{H}^j (\Obs_{a,b} ) ) \twoheadrightarrow H^{i+j} ( \Obs_{a,b}(E)),
$$
where $\mc{H}^j (\Obs_{a,b})$ refers to the cohomology sheaf.  

Recall that the cohomology sheaves
$$
\mc{H}^{4-4g + 2m + k}\left( \Obs_{2-2g, 2-2g + m  }\right) = 0
$$
for all $ k \le  0$ and $g > 0$ (this vanishing statement is part of Theorem \ref{theorem obstruction vanishing elliptic}).  It follows that $H^{4 -4 g + 2m}\left( \Obs_{2-2g,2-2g+m} (E) \right) = 0$, so that the Virasoro constraints hold. 
\end{proof}

\subsection{}
Finally, we will analyze how the Virasoro constraints manifest themselves at the level of correlators.  
\begin{corollary}
For any splitting $\br{F}$ of the Hodge filtration on $H^1(E)$, and any $\alpha_i \in H^\ast (\PV(E))[[t]]$, the correlators satisfy the following equations.
$$
\ip{1 \cdot t^{l} , \alpha_1, \ldots, \alpha_n } _{g,n+1}^{E, \br{F}} 
=  \sum_{i = 1}^n \binom{l + \op{HW}(\alpha) }{l}  \ip{ \alpha_1 , \ldots,t^{l - 1} \alpha_i , \ldots, \alpha_n  } _{g,n}^{E, \br{F}} .
$$
Here, $\op{HW}(\alpha)$ denotes the Hodge weight of $\alpha$, defined by $\op{HW}(\alpha) = i + k-1$ if $\alpha \in t^k\PV^{i,j}(E)$. Also, the sign $\pm$ in the second equality is obtained by the usual Koszul rule, which contributes a $(-1)^{\abs{\alpha_j}}$ each time we move $\d \zbar$ past some $\alpha_j$. 
\end{corollary}
\begin{proof}
This is immediate from the fact that the correlators only depend on the homotopy class of the quantization.  
\end{proof}

\section*{Appendix I : The obstruction group on an elliptic curve}

In this Appendix we will prove the following.

{
\renewcommand{\thetheorem}{\ref{theorem obstruction vanishing elliptic}}
\begin{theorem}
The cohomology sheaves $\mc{H}^i(\Obs_{a,b})$ are constant sheaves.  Further, 
$$
\mc{H}^{ i + 2 b  }(\Obs_{a,b}) = 0
$$
if $a \le 0$ and $(i,b) \le (0,0)$ in the lexicographical ordering on $\Z \times \Z$.  
\end{theorem}
\addtocounter{theorem}{-1}
}

In \cite{Cos11}, Chapter 5, it is shown that the obstruction complex for any quantum field theory on a manifold $M$ can be written in terms of the Chevalley cochains of an $L_\infty$ algebra in the symmetric monoidal category of $D_M$ modules.  Here, $D_M$ denotes the algebra of $\cinfty$ differential operators.  The category of sheaves on $M$ of modules over $D_M$ is a symmetric monoidal category, where the tensor product is as modules over the sheaf $\cinfty_M$ of smooth functions on $M$.

This $D_M$ $L_\infty$-algebra has, as underlying differential graded $D_M$-module, the sheaf of jets of fields (with a shift). In the case of BCOV theory on an elliptic curve $E$, the $D_M$ module of jets of fields is
$$
\til{\mf g} = J(\cinfty_E) [\d \zbar, \partial_{z} ] [[t]] [1].
$$
Here $J(\cinfty_E)$ is the bundle of jets of smooth functions on $E$; and $\d \zbar$, $\partial_z$ are both situated in cohomological degree $1$, and $t$ has cohomological degree $2$.

If we choose coordinates on a domain $U$ on $E$, then the sections of $\til{\mf g}$ are
$$
\cinfty(U) \otimes \C[[z,\zbar, \d \zbar, \partial_z, t ]][1].
$$
If $w$ and $\wbar$ indicate coordinates on $U$, then the flat connection on this jet bundle differs from the trivial connection for this trivialization by the one-form
$$
\d w \frac{\partial}{\partial z } + \d \br{w} \frac{\partial}{\partial \zbar }.
$$

\subsection{}
As explained in \cite{Cos11}, the $D_E$-module $\g$ is equipped with the structure of $L_\infty$ algebra (in the symmetric monoidal category of $D_E$-modules) which arises from the classical action functional $I$, and from the differential $Q$ on the space of fields.  This $L_\infty$ algebra is constructed so that there is an isomorphism of cochain complexes
$$
\left( \Ool(\Fields(E)), Q + \{I,-\} \right) \simeq C^\ast_{red}(\til{\g}) \otimes_{D_E} \Omega^2_E.
$$
On the right hand side of this equation, $C^\ast_{red}(\til{\mf{g}})$ refers to the reduced Chevalley-Eilenberg complex of $\g$, taken in the symmetric monoidal category of $D_E$ modules equipped where tensor product is over $\cinfty_E$.  Since, strictly speaking, $\til{\g}$ is a topological $D_E$ module, one must use continuous multilinear maps to define this reduced Chevalley-Eilenberg complex.  The resulting $D_E$-module $C^\ast_{red}(\til{\g})$ is then tensored, over $D_E$, with the right $D_E$ module of $2$-forms on $E$ to obtain the complex of local functionals.

Further, $C^\ast_{red}(\til{g})$ is automatically flat as a $D_E$-module.  This means that there is a quasi-isomorphism
$$
\left( \Ool(\Fields(E)), Q + \{I,-\} \right) \simeq C^\ast_{red}(\til{\g}) \otimes^{\mbb L}_{D_E} \Omega^2_E.
$$
where on the right hand side we are taking the derived tensor product.  

Recall that, for any $D_E$-module $M$, one can define the de Rham complex $\Omega^\ast(E,M)$, and one has 
$$
\Omega^\ast(E,M) [2] \simeq M \otimes^{\mbb L}_{D_E} \Omega^2_E. 
$$
Thus, we find a quasi-isomorphism
$$
\Ool(\Fields(E)) \simeq \Omega^\ast (E, C^\ast_{red}(\til{\g})) [2].
$$
This quasi-isomorphism allows us to produce a spectral sequence converging to the cohomology of $\Ool(\Fields(E))$, as follows. 

Let $\mc{H}^i_{red}(\til{\g})$ denote the $D_E$-module obtained as the cohomology of the complex of $D_E$ modules $C^\ast_{red}(\til{\g})$.   Then, there is a spectral sequence
$$
H^i_{dR}( E,  \mc{H}^j_{red}( \til{\g} ) ) \Rightarrow H^{i+j-2} ( \Ool(\Fields(E) ).
$$
Use of this spectral sequence allows one to translate results about the vanishing of the cohomology of $\til{\g}$ into results about the vanishing of the cohomology of $\Ool(\Fields(E))$.  This will be essential for our obstruction computations.  

\subsection{}
Let us now describe the $L_\infty$ algebra structure on $\mf g$.  This is the structure of an $L_\infty$ algebra in the category of $D_E$ modules.  Thus, all of the structures are, in particular, linear over the sheaf $\cinfty_E$.  This means we can describe the $L_\infty$ structure explicitly by describing it on the fibre.

The differential, acting on a fibre
$$
\til{\g}_0 = \C[[z,\zbar,\d \zbar, \partial_z, t]] [1]
$$ 
is 
$$
Q = \d \zbar \frac{\partial}{\partial \zbar} - t  \frac{\partial}{\partial z} \frac{\partial}{\partial (\partial_z)}.
$$
The first term in $Q$ is simply the jet of the $\dbar$ operator, acting on the jets of the Dolbeaut forms. The second term in $Q$ is the jet of the operator
$$
t \partial : \Omega^i(E, T E ) [[t]]\to \Omega^i(E)[[t]].
$$

Because the classical action functional $I$ for the BCOV theory does not involve any anti-holomorphic derivatives, the $L_\infty$ structure on $\til{\mf g}$ is linear over $\C[[\zbar, \d \zbar]]$.  Further, the $L_\infty$ structure will be such that the subcomplex
$\g \subset \til{g}_0$ defined by
$$
\g = \C[[z, t, \partial_z]] [1]  \subset \C[[z,t, \partial_z, \zbar, \d \zbar]] [1]
$$
is, in fact, a sub $L_\infty$ algebra.  Since the inclusion $\g \to \til{\g}_0$ is clearly quasi-isomorphic to $\til{g}_0$, it suffices to describe explicitly the $L_\infty$ structure on $\g$. 

\subsection{}
The $L_\infty$ structure maps
$$
l_k : \g^{\otimes k} \to \g
$$
are defined as follows.  Firstly, $l_k$ is zero unless there is precisely one $\partial_z$ among the inputs.  Then,
$$
l_k ( f_1 (z) t^{n_1} \partial_z , f_2(z) t^{n_2}, \ldots, f_k(z) t^{n_k} ) = \frac{\d}{\d z} \left( \prod f_i(z) \right) \int_{\mbar_{0,k+1} } \psi_1^{n_1} \ldots \psi_k^{n_k},
$$
for $f_i \in \C[[z]]$.  Note that this is zero unless $\sum n_i = k-2$.  

In particular,  $l_2$ is defined by
$$
l_2( f(z) \partial_z, g(z)) = \frac{\d}{\d z} \left( f(z) g(z) \right),
$$
and $l_2$ is zero on any other inputs.

These formulae simply arise by considering how Poisson bracket with the classical action functional $\{I,-\}$ acts on the space of local functionals $\Ool(\Fields(\E))$, and then translating this into the language of jets. 

\subsection{}
We are interested in quantizing the BCOV theory in a way compatible with the Hodge weight and dilaton axioms.  We thus need to incorporate both of these into our $L_\infty$ algebra $\g$.

Let us first consider Hodge weight. The vector space $\g$ is bigraded, by cohomology degree and by Hodge weight.  A typical element of $\g$ is of the form $t^k z^l \partial^m$, where $k,l \ge 0$ and $m \in\{0,1\}$.  The cohomological degree and Hodge weight of this are   
\begin{align*}
\text{Hodge weight} (t^k z^l \partial^m ) &= k + m - 1\\
\text{Cohomological degree} (t^k z^l \partial^m) &= 2 k + m - 1.
\end{align*}
One can easily check that $l_n$ is of cohomological degree $2-n$ and Hodge weight $0$.  Thus, $\g$ is a graded $L_\infty$ algebra, with grading by Hodge weight.

\subsection{}
Now let us consider the dilaton axiom.  Let us define the dilaton operator, as before, by
$$
\partial_{Dil} = \partial_{t \cdot 1 } - \op{Eu} : C^\ast_{red} (\g) \to C^\ast_{red}(\g). 
$$
Here $\partial_{t\cdot 1}$ is the operator of differentiating with respect to the element $t  \in \g$, and $\op{Eu}$ is the Euler derivation, characterized by the property that for $\alpha \in \g^\vee \subset C^\ast_{red}(\g)$,
$$
\op{Eu}(\alpha) = \alpha. 
$$
We are interested in the homotopy $a$-eigenspace of $C^\ast_{red}(\g)$ under the action of the derivation $\partial_{Dil}$.  This is described by the complex
$$
\left( C^\ast_{red}(\g) [\eps], \d_\g + \eps \left( \partial_{Dil} - a\right) \right).  
$$
Here $\d_\g$ refers to the differential on $C^\ast_{red}(\g)$, and $\eps$ is a parameter of cohomological degree $1$ and square zero.  

We will let
$$
H^i_{a,b} (\g) \subset H^i \left( C^\ast_{red}(\g) [\eps], \d_\g + \eps \left( \partial_{Dil} - a\right) \right)
$$
be the part in Hodge weight $b$.   This group is of crucial importance, because of the spectral sequence
$$
H^i (E, H^j_{a,b}(\g)) \Rightarrow H^{i+j -2} (\Obs_{a,b}(E))
$$
where $\Obs_{a,b}(E)$, as before, is the part of  obstruction-deformation complex controlling quantizations of BCOV in dilaton weight $a$ and Hodge weight $b$.

\subsection{}
The following is the our main vanishing result concerning $\g$; this implies proposition 
\label{proposition_obstruction_elliptic}.
\begin{proposition}
$H^{i+2b}_{a,b}(\g) = 0$ if $a \le 0$ and $(i,b) \le (2,0)$  in the lexicographical ordering on $\Z  \times \Z$. 
\label{proposition lie algebra vanishing}
\end{proposition} 

\begin{proof}
As a first step, we will construct an operator $\til{\partial}_{Dil}$ on $C^\ast_{red}(\g)$ which is homotopic to $\partial_{Dil}$, and which will be more useful for calculation.

The operator $\til{\partial}_{Dil}$ is a derivation of $C^\ast_{red}(\g)$.  Thus, it is determined by its Taylor components
$$
\til{\partial}_{Dil}(n) : \Sym^n (\g[1] ) \to \g[1].
$$
These are defined by
\begin{align*}
\til{\partial}_{Dil}(1) (z^k  )  &= k z^k \\
\til{\partial}_{Dil}(1) (z^k  t^l \partial_z^m )  &= 0 \text{ unless } l = m = 0 \\
 \til{\partial}_{Dil}(n) (\alpha_1, \ldots, \alpha_n) &= l_{n+1} ( \alpha_1, \ldots, \alpha_n , z \partial_z ) \text{ for } \alpha_i \in \g, \ n \ge 2 \\
\til{\partial}_{Dil}(0) &= 0.
\end{align*}
\begin{lemma}
The operator $\til{\partial}_{Dil}$ on $C^\ast_{red}(\g)$ is a cochain map, and is cochain homotopic to $\partial_{Dil}$.
\end{lemma}
\begin{proof}
The cochain homotopy is the operator $\frac{\partial}{\partial (z \partial_z)}$. This is the derivation whose zero${}^{th}$ Taylor component is the element $z \partial_z \in \g$, and whose remaining Taylor components vanish.  
\end{proof}

Next, we will construct a smaller $L_\infty$ algebra $\g'$ which is quasi-isomorphic to $\g$, and on which the modified dilaton operator $\til{\partial}_{Dil}$ acts.   Explicitly, we let
$$
\g' = \C[[z]] \oplus \C[[t]] \partial_z  \subset \g.
$$
As before, of course, $z^k$ is in cohomological degree $-1$, and $t^l \partial_z$ is in cohomological degree $2 l$.   It is clear that $\g'$ is a sub $L_\infty$ algebra of $\g$, and that the inclusion $\g' \to \g$ is a quasi-isomorphism.  Further, the modified dilaton operator $\til{\partial}_{Dil}$ acts on $C^\ast_{red}(\g')$ in a way compatible with the map $C^\ast_{red}(\g) \to C^\ast_{red}(\g')$.

\begin{definition}
Let $S : \g' \to \g'$ be the \emph{scaling operator}, defined by
\begin{align*}
S (z^k ) &= k z^k \\
S( t^l \partial_z ) &= - t^l \partial_z.
\end{align*}
Let $\partial_S : C^\ast_{red}(\g') \to C^\ast_{red}(\g')$ be the derivation associated to $S$.
\end{definition}

\begin{lemma}
When restricted to $C^\ast_{red}(\g')$, the modified dilaton operator $\til{\partial}_{Dil}$ coincides with $\partial_S$.
\end{lemma}
\begin{proof}
The higher Taylor components of $\til{\partial}_{Dil}$ vanish on $\g'$, and the first Taylor component of $\til{\partial}_{Dil}$ is $S$.
\end{proof}

The $L_\infty$ algebra $\g'$ is bigraded; one grading by Hodge weight, and the other by the eigenvalues of the scaling operator $S$.  We will call the second grading the scaling weight.  The two gradings are as follows:
\begin{align*}
z^k & \text{ has Hodge weight } -1 \text{ and scaling weight k } \\
t^k \partial_z & \text{ has Hodge weight } k \text{ and scaling weight -1 }. 
\end{align*}
This bigrading induces a bigrading on the Chevalley-Eilenberg cohomology of $\g'$. We will let $H^k_{a,b}(\g')$ denote the reduced Chevalley-Eilenberg cohomology group in scaling weight $a$ and Hodge weight $b$. 

\begin{lemma}
$$
H^k_{a,b} (\g' ) = H^k_{a,b}(\g)
$$
where $H^k_{a,b}(\g)$ refers, as before, to the reduced cohomology groups of $\g$ in dilaton weight $a$ and Hodge weight $b$.
\end{lemma}
\begin{proof}
We have seen that $\g' \simeq \g$, so that $H^k(\g') = H^k(\g)$.  Further, the scaling weight operator on $H^k(\g')$ coincides with the dilaton operator on $H^k(\g)$.
\end{proof}

\subsection{}
Thus it remains to show the following.
\begin{proposition}
$$
H^{i + 2 b}_{a,b}(\g') = 0 
$$
if $a \le 0$ and $(i,b) \le (2,0)$ in the lexicographical ordering on $\Z \times \Z$.  
\end{proposition}

The rest of this section will be concerned with the proof of this proposition. 

 We will compute the (appropriately completed) Lie algebra homology of $\g'$ (because the notation here is more transparent).  Because cohomology is defined using the space of continuous multilinear maps, the cohomology is simply the dual of the homology.   Although we are using homology, our conventions are such that the differential is of degree $+1$.  If $H_k^{a,b}(\g')$ refers to the homology in degree $k$, scaling weight $a$ and Hodge weight $b$, we have
$$
H_k^{a,b}(\g')^\vee = H^{-k}_{-a,-b} (\g').
$$

\subsection{}
Let us change notation, and set $e_k = z^k$, and $\delta_k = t^k \partial_z$.   The Lie algebra chain complex of $\g$ can be identified with the algebra of power series in the variables $\delta_k,e_l$, where $k,l \ge 0$:
$$
A = \C[[\delta_k, e_l]].
$$

Here, $e_l$ has cohomological degree $-2$, and each $\delta_k$ has cohomological degree $2k-1$.  Also, $e_l$ has Hodge weight $-1$ and scaling weight $l$, whereas $\delta_k$ has Hodge weight $k$ and scaling weight $-1$.      

It will be convenient to arrange these three degrees into a integer vector which we call the tridegree; an element has tridegree $(i,a,b)$ if it has cohomological degree $i$, scaling weight $a$ and Hodge weight $b$.

\subsection{}
The differential on $A$ will be defined in terms of a sequence of auxiliary differential operators $\Phi_n$ on $A$.   We define $\Phi_n$ by the formula
$$
\Phi_n = \sum_{k_1,\ldots,k_n} (k_1 + \cdots + k_n) e_{\sum k_i - 1} \frac{\d}{\d e_{k_1} } \cdots \frac{\d}{\d e_{k_n}}.  
$$ 
The differential on $A$ is defined by the formula
$$
\d = \sum \frac{\d}{\d \delta_k} \Phi_{k+1}.
$$

Our aim is to show that $H_{i+2b}^{a,b} (A,\d ) = 0$ if $a \ge 0$ and $(i,b) \ge (-2,0)$ in the lexicographical ordering on $\Z \times \Z$.  

In other words, we aim to show that $A$ has no cohomology in tridegree $(i+2a, b, a)$ if $a \ge 0$ and either $i > -2$, or $i = -2$ and $b \ge 0$.  

\subsection{}
There is a spectral sequence converging to the cohomology of $A$, whose first term is given by the cohomology of $A$ with the differential
$$
\d_0 = \frac{\d}{\d \delta_0} \Phi_{1} = \frac{\d}{\d \delta_0} \sum_{k\ge 0} (k+1) e_{k} \frac{\d}{\d e_{k+1}}
$$
Thus, the first step is to calculate the cohomology of $A$  with respect to $\d_0$.

\subsection{$\d_0$ cohomology.}
To compute the $\d_0$ cohomology, we need an auxiliary spectral sequence.  If $k > 0$,  we will let $F^k A$ be the subspace of $A$ spanned by elements divisible by $e_l$, for some $l \ge k$.  Let $F^0 A = A$.   Let
$$
\Gr^k A = F^k A / F^{k+1} A.
$$
Thus,
$$
\Gr^k A = e_k \C[[e_0,e_1,\ldots, e_k , \delta_0,\delta_1,\ldots, ]],
$$
if $k > 0$.  Whereas,
$$
\Gr^0 A = \C[[e_0, \delta_0,\delta_1,\ldots]].
$$
Note that the map $\Phi_1$ sends $F^k A \to F^{k-1} A$.  Thus, it descends to a map
$$
\Gr^k \Phi_1 : \Gr^k A \to \Gr^{k-1} A.
$$  
\begin{lemma}  If $k > 0$ then $\Gr^k \Phi_1$ is surjective. The kernel of $\Gr^k \Phi_1$ is canonically isomorphic to the graded vector space
$$
e_k^2 \C[[ e_0,\ldots, e_k, \delta_0,\delta_1,\ldots]].
$$
\end{lemma}

We will let 
$$
A_k = \op{Ker} (\op{Gr}^k \d_0 ) / \op{Im}( \op{Gr}^{k+1} \d_0 ),
$$
if $k \ge 0$.  This lemma implies that, if $k > 0$,
$$
A_k = \delta_0 e_k^2 \C[[ e_0,\ldots, e_k, \delta_1, \delta_2, \ldots]].
$$
Further, 
$$
A_0 = \delta_0 \C[[e_0, \delta_1, \delta_2,\ldots]].
$$
 This lemma implies that the cohomology of the operator
$$
\Gr^\ast \d_0 = \frac{\d}{\d \delta_0} \Gr^k \Phi_1 : \oplus \Gr^k A \to \oplus \Gr^k A
$$
is canonically isomorphic to the graded vector space
$$
\oplus_{k \ge 0} A_k.
$$

\subsection{Cohomology with respect to $\d_1$.}

The next term in our spectral sequence is given by the operator
$$
\d_1 = \frac{\d}{\d \delta_1} \Phi_2.
$$ 
Recall that
$$
\Phi_2 = \sum (k_1 + k_2) e_{k_1 + k_2 - 1} \frac{\d}{\d e_{k_1} } \frac{\d}{\d e_{k_2}}.
$$

If $A_k$ is the $k$'th graded piece of the cohomology of $\d_0$, as defined above, the differential $\d_1$ maps $A_k$ to itself.    Thus, we will investigate the $\d_1$ cohomology of $A_k$. 

The operator $\d_1$ on $A_0$ is zero.   If $k > 0$, let us define a filtration on $A_k$ by saying that $F^l A_k$ is the subspace spanned by those $f$ which are divisible by $e_1^l$.  Note that the differential $\d_1$ maps $F^l A_k$ to $F^{l-1} A_k$.   Thus, we can consider the operator
$$
\Gr^l \d_1 : \Gr^l A_k \to \Gr^{l-1} A_k.
$$

Let 
$$
W = \sum_{n = 0}^k e_n \frac{\d}{\d e_n } : A_k \to A_k.
$$
Note that $W$ is diagonal on the natural basis of each $A_k$; and (as $k > 0$), the eigenvalues of $W$ are all non-zero.

It is easy to see that
$$
\Gr^l \d_1 = \frac{\d}{\d \delta_1} \frac{\d}{\d e_1} W : \Gr^l A_k \to \Gr^{l-1} A_k.
$$
This implies that, if $k > 1$, the cohomology of $A_k$ with respect to $\d_1$ is isomorphic to 
$$
B_k = \delta_1 \delta_0 e_k^2 \C[[\delta_2, \delta_3, \ldots, e_0, e_2, e_3 ,\ldots, e_k]].
$$
If $k = 1$, the cohomology of $A_1$ with respect to $\d_1$ is 
$$
B_1 = \delta_1 \delta_0 e_1^2 \C[[e_0, \delta_2, \delta_3, \ldots ]].
$$

If $k = 0$, the operator $\d_1$ is zero on $A_0$.  Thus, we set 
$$
B_0 = A_0 = \delta_0 \C[[e_0, \delta_1,\ldots]].
$$

\subsection{Completion of proof}
Recall that $\delta_k$ is in tridegree $(2k-1,-1,k)$, and $e_l$ is in tridegree $(-2,l, -1)$.  Recall further that we are interested in showing that there are no classes in tridegree $(i + 2b, a, b)$ when $(i,b) \ge (-2,0)$ and $a \ge 0$.  Note that this amounts to the following: either $i = -2$, $a \ge 0$ and $b \ge 0$ or when $i > -2$, $a \ge 0$ and $b \in \Z$.   

If $k = 0$, then $B_0$ is concentrated in tridegrees $(i+2b, i, b)$ when $i \le -1$.  Thus, $B_0$ contains no relevant classes.  

If $k > 1$, then $B_k$ is concentrated in tridegrees $(i+ 2 b, a, b)$ with $i \le -2$.  Further, the subspace of $B_k$ of tridegree $(-2+2b,a,b)$ consists of elements of the form $\delta_1 \delta_0 e_k^2 f$, where $f$ is in $\C[[e_0,e_2,\ldots,e_k]]$.  An element of this form is in tridegree $(-2+2b , a, b)$ where $a > 0$ and $b < 0$.  Such an element can not contribute to the cohomology groups of interest. 

The final case is when $k = 1$.  The subspace of $B_1$ of tridegree $(-2+2a, b, a)$ is spanned by the elements $\delta_1 \delta_0 e_1^2 e_0^m$.  This element is in tridegree $(-2m-4 , 0, -m-1)$, and so can not contribute to the cohomology groups of interest.  

This completes the proof. 
\end{proof}

\def\cprime{$'$}

\end{document}